\documentclass[hidelinks,onefignum,onetabnum]{siamart220329}

\usepackage[utf8]{inputenc}
\usepackage{bm}
\usepackage{mathrsfs}
\usepackage{amsmath}
\usepackage{amsfonts}
\usepackage{amssymb}
\usepackage{multirow}
\usepackage{dsfont}
\usepackage{caption}
\usepackage{subfigure}
\usepackage{enumitem}
\usepackage{booktabs}
\usepackage{comment}

\usepackage[english]{babel}
\makeatletter
\newcommand{\mylabel}[2]{#2\def\@currentlabel{#2}\label{#1}}
\makeatother

\newcommand{\mcS}{\mathcal{S}}

\newcommand{\mcD}{\mathcal{D}}
\newcommand{\mcE}{\mathcal{E}}

\newcommand{\mcA}{\mathcal{A}}

\newcommand{\RR}{\mathrm{RR}}
\newcommand{\DA}{\mathrm{DA}}

\usepackage{xcolor}

\def \bb{\mathbf{b}}

\def \phib{\bm{\phi}}

\def \Cb{\bm{C}}
\def \eb{e}

\newcommand{\calB}{{\mathcal{B}}}

\def\E{\mathbb{E}}

\def\calN{\mathcal{N}}

\def\calC{\mathcal{C}}

\def\spaceY{\mathbb{Y}}

\newtheorem{assumption}[theorem]{Assumption}

\newtheorem{example}[theorem]{Example}
\newtheorem{remark}[theorem]{Remark}
\newcommand{\creflastconjunction}{, and~}
\newsiamremark{hypothesis}{Hypothesis}
\crefname{hypothesis}{Hypothesis}{Hypotheses}

\usepackage{lineno}

\usepackage{lipsum}
\usepackage{amsfonts}
\usepackage{graphicx}
\usepackage{epstopdf}
\usepackage{algorithmic}
\ifpdf
  \DeclareGraphicsExtensions{.eps,.pdf,.png,.jpg}
\else
  \DeclareGraphicsExtensions{.eps}
\fi

\headers{}{J. Wang, F. Lu and Y. Yu}

\title{When Rough Data Helps: A Phase Transition in Convergence Rates for Kernel Recovery in Integral Operators
\thanks{Submitted to the editors DATE.\funding{The work of J. Wang is supported by the Natural Science Foundation of Hubei Province under grant 2026AFB322. The work of F. Lu is partially supported by the National Science Foundation under grants DMS-2238486 and DMS-2511283. Y. Yu would like to acknowledge the support from the National Science Foundation under grant DMS-2436624 and the Air Force Office of Scientific Research under grant FA9550-22-1-0197.}}}

\author{
Jihong Wang\thanks{School of Mathematics and Statistics, Huazhong University of Science and Technology, Wuhan 430074, China;
Hubei Key Laboratory of Engineering Modeling and Scientific Computing, Huazhong University of Science and Technology, Wuhan 430074, China.}
\and Fei Lu\thanks{Department of Mathematics, Johns Hopkins University, Baltimore, MD 21218, USA.}
\and Yue Yu\thanks{Department of Mathematics, Lehigh University, Bethlehem, PA 18015, USA.}
}

\usepackage{amsopn}

\begin{document}

\maketitle
\begin{abstract}
Learning kernels in operators from data is a fundamental task that arises in nonlocal continuum mechanics, operator learning, and interacting particle systems. A central question is how the roughness of input data impacts the accuracy of kernel recovery. We quantify the roughness of the input data via its spectral decay exponent and analyze how it determines the degree of ill-posedness of the inverse problem and, consequently, the convergence rates of the Tikhonov-regularized estimator in the small-noise limit. Within this framework, we identify a phase transition between an under-rough regime, in which rougher data improves recovery, and an over-rough regime, in which further roughening leads to slower rates. These theoretical findings are supported by numerical experiments ranging from idealized settings to more realistic configurations, with quantitative agreement in the former and broad consistency of the main trends in the latter.
\end{abstract}




\tableofcontents

\section{Introduction}

Integral operators with kernel functions are fundamental models for nonlocal interactions and long-range dependence in scientific computing, arising in nonlocal mechanics and peridynamics~\cite{silling2000reformulation,you2021data,xu2021machine,lu2024nonparametric,d2017nonlocal,suzuki2023fractional,du2018peridynamic,burkovska2022optimization,geng2025end}, interacting particle and agent-based systems~\cite{bongini2017inferring,LZMT19,lu2021stochastic,lang2022learning,du2020mathematics,tang2024identifiability}, and linear system identification~\cite{pillonetto2010new}. In these settings, the kernel function encodes essential physical information \cite{xu2022machine,askari2008peridynamics,mengesha2013analysis,huang2022unified}: material constitutive laws in peridynamics, pairwise interaction rules in multi-agent systems, and impulse responses in dynamical systems. Its accurate determination from data is therefore fundamental to predictive modeling.
A natural inverse problem arises: given noisy input-output observations $\mcD = \{(u_m, {b}_m)\}_{m=1}^M$ of an integral operator
\begin{equation*}
    R_\phi[u](x) := \int_{\mcS} \phi(s)\, g[u](x,s)\, ds = b(x), \quad x\in\Omega,
\end{equation*}
to learn the unknown kernel function $\phi:\mcS\to\mathbb{R}$. Here $\mcS$, $\Omega\subset\mathbb{R}$ are bounded domains representing the kernel support and the observation domain, respectively, and $g[u](x,s)$ is a known interaction term dependent on the input function $u$.

There is growing empirical evidence that the properties of the input data, in particular their roughness, affect kernel recovery accuracy. In peridynamics, smooth loading conditions fail to excite high-frequency deformation modes, limiting the resolution of the recovered material kernel~\cite{wang2026monotone}; in system identification, broadband excitation is preferred over narrowband precisely because it probes a wider range of the system's frequency response~\cite{pillonetto2010new}. These observations suggest that rougher data should improve kernel recovery, but a quantitative understanding is lacking. The broad purpose of this paper is to provide a quantitative answer to the question: \textit{how does data roughness affect the accuracy of kernel estimators, and is there an optimal level of roughness?} Such an answer would guide the design of training data across the applications above, and in kernel learning problems more broadly.

\textbf{Related work.}
Driven by the growing use of machine learning methods for operator learning in scientific computing~\cite{kovachki2023neural,geng2025parallel,lu2021learning,guo2024ib}, recent work has begun to investigate how the quality of input training data affects learning accuracy. De~Hoop et al.~\cite{dehoop2023convergence} prove that rougher, less smooth training inputs lower the sample complexity of learning linear operators, and Boull\'{e} et al.~\cite{boulle2023elliptic} show that more oscillatory forcing functions improve the recovery of Green's functions. Lanthaler et al.~\cite{lanthaler2022error} establish error estimates for DeepONet operator learning that depend on the regularity of the input data, through the eigenvalue decay of the input covariance operator. However, these works address the problem of learning the operator itself (i.e., the map from input to output), rather than recovering a kernel function embedded inside an integral operator. The two tasks differ not only in their target but in their difficulty: the input-output map can be approximated directly from the observed input-output pairs, whereas the kernel $\phi$ is never observed directly and enters the data only through the integral operator it defines. Recovering it is therefore an ill-posed inverse problem \cite{donatelli2025basic,li2024preconditioned}.

Methods for this kernel recovery problem have been developed in interacting particle systems and nonlocal mechanics, where the unknown kernel is expanded in a finite set of basis functions and estimated by regularized least squares~\cite{bongini2017inferring,LZMT19,lu2021stochastic,lang2022learning,you2021data,xu2021machine}. On the theoretical side, Lu et al.~\cite{lu2024nonparametric, LLA22} show that kernel learning is an ill-posed inverse problem, characterize a function space of identifiability, and propose a data-adaptive RKHS Tikhonov regularization on that space whose estimator converges as the data mesh refines, and Zhang et al.~\cite{zhang2025minimax} derive minimax-optimal convergence rates for kernel recovery, quantifying the best achievable error as the sample size grows. Both works study how the amount of data, whether the number of samples or the observation resolution, controls the recovery error; neither addresses how the roughness of the input data affects the recovered kernel. The same holds for the broader convergence-rate theory of regularized ill-posed estimation, spanning statistical regression~\cite{CD07,blanchard2018optimal,hall2007methodology,cai2012minimax,yuan2010reproducing,crambes2009smoothing} and classical inverse problems~\cite{tikhonov1963solution,engl1996regularization,kirsch2021introduction,hansen2010discrete,cavalier2008nonparametric,bissantz2007convergence,knapik2011bayesian}: these works take the spectral decay of the forward operator as a known, given property and tie the rate to it together with a source condition, but none consider how that decay depends on the smoothness of the data. In kernel recovery, by contrast, the forward operator is built from the input data, so its spectral decay, and hence the degree of ill-posedness, is determined by the roughness of the inputs. We analyze how the convergence rate of the kernel estimator depends on the data roughness, using the small-noise analysis framework of Lang and Lu~\cite{lang2023small}, which establishes the convergence rate of the regularized estimator as the noise vanishes. 

\textbf{Main results.}
We characterize data roughness by the Fourier spectral decay rate $\alpha$ of the input functions: smaller $\alpha$ corresponds to slower decay of the Fourier coefficients and hence rougher, spectrally richer data. The operator's eigenvalue decay rate $\gamma$ depends jointly on $\alpha$ and the interaction term $g$. This map from $\alpha$ to $\gamma$ can be derived analytically for a broad class of translation-invariant operators (Lemma~\ref{lem:spectrum_general}) and determined numerically in more general settings. 

To recover the kernel, we use the Tikhonov estimator with a regularization operator taken from the class $\calC=(\mcA^\top\mcA)^{-s}$ ($s\ge0$); the exponent $s$ controls how strongly the penalty adapts to the operator spectrum, and recovers ridge regression at $s=0$ and data-adaptive RKHS regularization at $s=1$.
The convergence rate of the kernel error, obtained by decomposing the estimation error into bias and variance in the spectral domain and optimizing the regularization parameter, depends on the operator spectral decay $\gamma$ (and hence on the data roughness $\alpha$) together with the smoothness of the target kernel itself. We quantify the latter by a smoothness exponent $\beta>0$ specified by a source condition on the kernel's coefficients in the operator's eigenbasis (Assumption~\ref{assump:source}); larger $\beta$ corresponds to a smoother target. 
Our main result (Theorem~\ref{theorem:small_noise}) establishes these rates explicitly for every $s\ge0$ and reveals a phase transition in the convergence rates at a critical spectral decay $\gamma^*$ that maximizes the convergence rate. The rate formulas for two commonly used estimators are stated in the left panel of Fig.~\ref{fig:rate_vs_gamma}, while the right panel plots the convergence rate $p$ as a function of $\gamma$, where the peak at $\gamma^*$ is clearly visible.

\begin{figure}[!t]
\makebox[\textwidth][c]{%

\begin{minipage}[t]{0.50\textwidth}
    \vspace{0pt}
    \small
    \textbf{Theorem~\ref{theorem:small_noise}} (Convergence rates, $\sigma\to 0$)

    \medskip
    \textit{(i) Ridge regression} ($\calC=I$):
    \vspace{-4pt}
    \[
    \E\bigl[\|\hat{\phi}_\lambda - \phi_*\|^2 | \mcA\bigr] \lesssim
    \begin{cases}
    \sigma^{\frac{4\beta}{2\beta+\gamma+1}}, & \beta < \gamma,\\[2pt]
    \sigma^{\frac{4}{3+1/\gamma}}, & \beta > \gamma.
    \end{cases}
    \]
    \vspace{-10pt}

    \textit{(ii) Data-adaptive RKHS} ($\calC\!=\!(\mcA^\top\! \mcA)^\dagger$):
    \vspace{-4pt}
    \[
    \E\bigl[\|\hat{\phi}_\lambda - \phi_*\|^2 | \mcA\bigr] \lesssim
    \begin{cases}
    \sigma^{\frac{4\beta}{2\beta+\gamma+1}}, & \beta < 2\gamma,\\[2pt]
    \sigma^{\frac{8}{5+1/\gamma}}, & \beta > 2\gamma.
    \end{cases}
    \]
    \vspace{-6pt}

\end{minipage}%

\hspace{0.03\textwidth}%
\begin{minipage}[t]{0.47\textwidth}
    \vspace{0pt}
    \includegraphics[width=\textwidth]{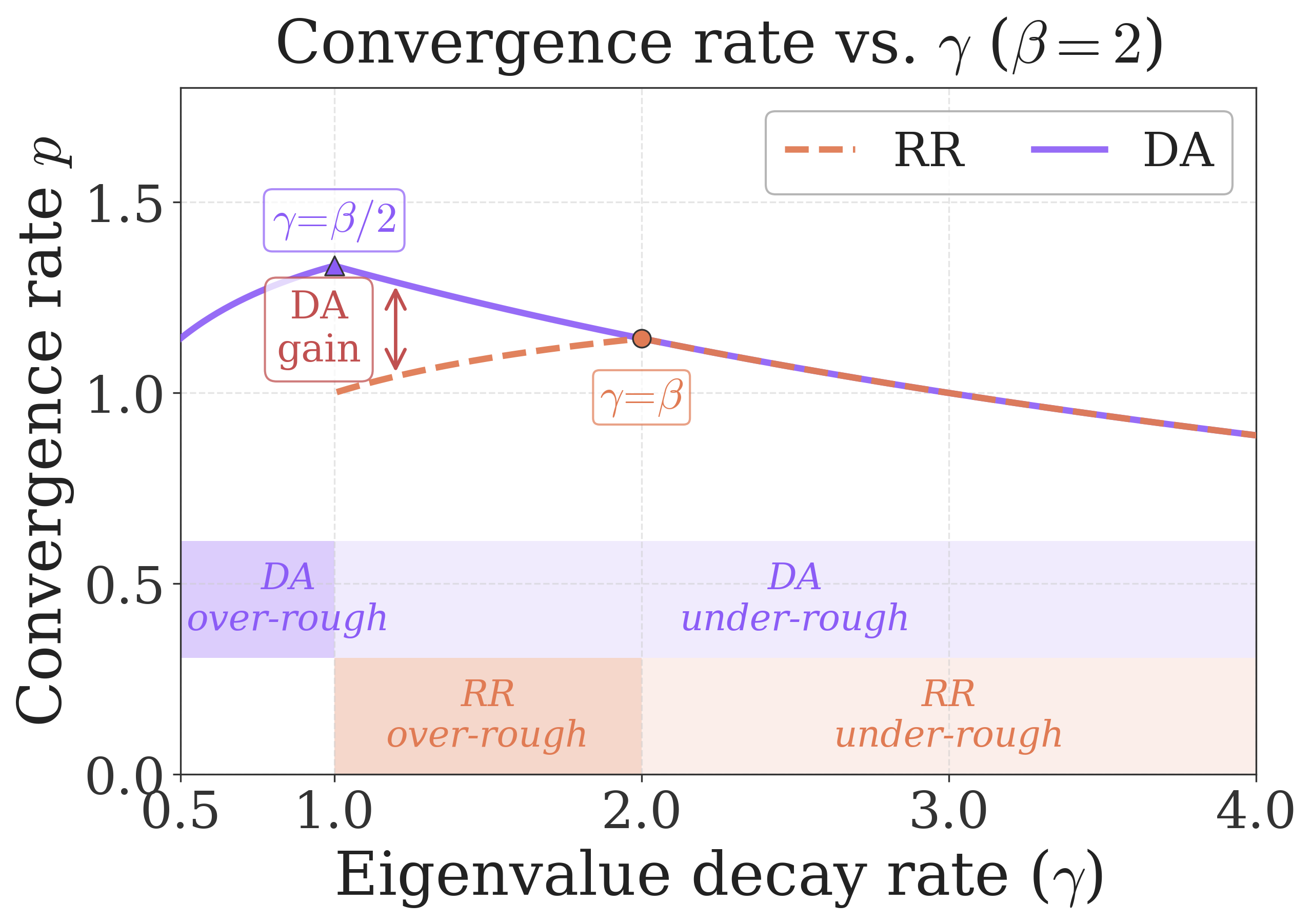}
\end{minipage}%

}%
\caption{(Left) Main convergence rates from Theorem~\ref{theorem:small_noise} for RR and DA regularizations: the kernel error scales as $\mathcal{O}(\sigma^p)$ with exponent $p$ depending on the operator spectral decay rate $\gamma$ and kernel smoothness  $\beta$. (Right) Visualization of the convergence rate $p$ vs.\ $\gamma$ at fixed $\beta=2$, showing the phase transition between the under-rough and over-rough regimes, an optimal $\gamma^*$ making the convergence rate maximal exists.}
    \label{fig:rate_vs_gamma}
\vspace{-2em} 
\end{figure}


\textbf{Contributions.}
The main contributions of this paper are as follows.

(1) We identify data roughness as a key factor impacting the accuracy of kernel recovery in integral operators. We characterize roughness through the spectral decay rate $\alpha$ of the input data, connect it to the spectral decay rate $\gamma$ of the data-induced forward operator, and derive small-noise convergence rates for the Tikhonov estimator.

(2) We discover a phase transition in the convergence rates between an under-rough regime, where rougher data improves recovery, and an over-rough regime, where further roughening becomes detrimental. This identifies an optimal data roughness that balances spectral richness against the smoothness of the target kernel, providing a principled guideline for training data design in kernel learning. We further show that a larger regularization exponent $s$ delays the over-rough regime and sharpens the convergence rate within it.

(3) We design a series of numerical experiments to validate the theoretical predictions, progressively relaxing the assumptions from idealized diagonal settings where the theory applies directly, to shared-eigenbasis integral operators and general non-aligned configurations that go beyond the theoretical framework, and finally to neural-network-based kernel learning methods, a practical setting outside our theoretical scope where similar data-roughness phenomena emerge.

\textbf{Notation.}
Table~\ref{tab:notation} summarizes the main notations used throughout the paper.
\begin{table}[htbp]
\centering
\footnotesize
\setlength{\tabcolsep}{4pt}
\caption{Summary of notation.}
\label{tab:notation}
\begin{tabular}{@{}ll@{\quad}ll@{}}
\toprule
Symbol & Description & Symbol & Description \\
\midrule
$\phi$, $\phi_*$, $\hat{\phi}_\lambda$ & kernel; true kernel; Tikhonov estimator & $s$ & regularization exponent  \\
$u$, $b$ & input function; output response & $\lambda$ & regularization parameter \\
$g[u]$ & interaction term & $d_{\rm eff}^{(s)}(\lambda)$ & effective dimension \\
$\varepsilon$, $\sigma$ & noise; noise level & $\tilde\beta$ & kernel coefficient decay rate \\
$\mcA$ & forward operator & $\beta$ & kernel smoothness exponent \\
$\mcS$, $\Omega$ & kernel domain; output domain & $\gamma$ & spectral decay rate of $\mcA^\top \mcA$ \\
$\delta$ & kernel interaction radius & $\alpha$ & data spectral decay rate \\
$\lambda_k$, $\psi_k$ & eigenvalues/functions of $\mcA^\top \mcA$ & $\gamma^*,\,\alpha^*$ & optimal $\gamma,\alpha$ at phase transition \\
$c_k$ & eigen-coefficient $\langle\phi,\psi_k\rangle$ & $K,\,K_\phi$ & input / kernel Fourier truncation \\
$M$ & number of training pairs $(u_m,b_m)$ & $\calC$ & regularization operator  \\
$N$ & dimension of the discretized kernel space & RR & ridge regression ($\calC{=}I$) \\
$J$ & number of observation points per output & DA & data-adaptive RKHS ($\calC{=}(\mcA^\top\mcA)^\dagger$) \\
\bottomrule
\end{tabular}
\end{table}

The rest of the paper is organized as follows. Section~\ref{sec:math_framework} introduces the mathematical framework, including the problem setting, the Tikhonov regularized estimator, and the spectral characterization of data roughness. Section~\ref{sec:errbd} presents the convergence analysis in the small-noise limit and states the main theoretical results. Section~\ref{sec:numerics} provides numerical verification across diagonal, shared-eigenbasis, general, and neural-network-based settings. Section~\ref{sec:conclusion} concludes with a discussion and outlines directions for future work. 

\section{Mathematical framework}\label{sec:math_framework}

\subsection{Problem setting}\label{subsec:problem_setting}
We consider the inverse problem of learning an unknown kernel function $\phi: \mathcal{S} \to \mathbb{R}$ governing the integral operator
\begin{equation}\label{eq:model}
    R_\phi[u](x) := \int_{\mathcal{S}} \phi(s) g[u](x,s) \, ds = b(x), \quad x \in \Omega,
\end{equation}
where $u$ is the input function, $b$ is the output response, and $g[u](x,s)$ is a known interaction term dependent on $u$. Given $M$ training pairs $\{(u_m, b_m)\}_{m=1}^M$ with $b_m = R_\phi[u_m] + \varepsilon_m$ and $\varepsilon_m$ the observation noise, we collect all samples into the operator equation
\begin{equation} \label{eq:dis_model}
    \mathbf{b} = \mcA \phi + \boldsymbol{\varepsilon},
\end{equation}
where $\mathbf{b} = (b_1,\dots,b_M)$, $\boldsymbol{\varepsilon} = (\varepsilon_1,\dots,\varepsilon_M)$, and the forward operator $\mcA: L^2(\mathcal{S}) \to \spaceY$ is the bounded linear map stacking the $M$ per-sample maps
\begin{equation}\label{eq:A_def}
    (\mcA\phi)_m(x) := R_\phi[u_m](x) = \int_{\mathcal{S}} \phi(s)\, g[u_m](x, s)\, ds, \quad m=1,\dots,M,\ x\in\Omega,
\end{equation}
into the data Hilbert space $\spaceY := L^2(\Omega)^{\oplus M}$, the $M$-fold direct sum of $L^2(\Omega)$ with inner product $\langle\mathbf{f},\mathbf{g}\rangle_\spaceY := \frac{1}{M}\sum_{m=1}^M \int_\Omega f_m(x) g_m(x)\,dx$, averaged over the $M$ samples so that the empirical normal operator $\mcA^\top\mcA$ converges to a population operator as $M\to\infty$ rather than growing linearly in $M$. Its adjoint $\mcA^\top: \spaceY \to L^2(\mathcal{S})$ is defined in the standard way.

Because each observed output is a function on the continuum $\Omega$ rather than a finite set of measurements, we model the noise as $\boldsymbol{\varepsilon} = \sigma\,\dot{\mathbf{W}}$, the function-space counterpart of i.i.d.\ Gaussian noise, where $\sigma>0$ is the noise level and $\dot{\mathbf{W}}$ is a Gaussian white-noise process on $\spaceY$ independent of $\mcA$; that is, $\boldsymbol{\varepsilon}$ is a centered Gaussian random linear functional on $\spaceY$ with covariance
\begin{equation}\label{eq:noise_cov}
\E\bigl[\langle\boldsymbol{\varepsilon},\mathbf{f}\rangle\langle\boldsymbol{\varepsilon},\mathbf{g}\rangle \,\big|\, \mcA\bigr] = \frac{\sigma^2}{M} \langle \mathbf{f},\mathbf{g}\rangle_\spaceY, \qquad \forall\,\mathbf{f},\mathbf{g}\in\spaceY,
\end{equation}
abbreviated as $\E[\boldsymbol{\varepsilon}|\mcA]=0$ and $\E[\boldsymbol{\varepsilon}\boldsymbol{\varepsilon}^\top|\mcA]=\tfrac{\sigma^2}{M} I$ with $I$ the identity on $\spaceY$. Here, $\langle\cdot,\cdot \rangle$ denotes the dual operation. The prefactor $1/M$ comes from the sample averaging built into $\langle\cdot,\cdot\rangle_\spaceY$.
Since $M$ is held fixed in the small-noise analysis where $\sigma\to0$, the factor $1/M$ enters the multiplicative constants and don't affect the convergence rates. Since $\boldsymbol{\varepsilon}$ does not belong to $\spaceY$, neither does $\mathbf{b}$, equation~\eqref{eq:dis_model} is read in the weak sense: for every $\mathbf{f}\in\spaceY$,  $\langle\mathbf{b},\mathbf{f}\rangle = \langle\mcA\phi,\mathbf{f}\rangle_\spaceY + \langle\boldsymbol{\varepsilon},\mathbf{f}\rangle$. 

Both the regularized estimator below and the bias-variance analysis are stated in terms of the spectral structure of the normal operator $\mcA^\top \mcA: L^2(\mathcal{S}) \to L^2(\mathcal{S})$, which is compact and positive semi-definite. Its eigenvalues $\{\lambda_k\}_{k\ge 1}$ form a non-negative sequence converging to zero, and the corresponding eigenfunctions $\{\psi_k\}_{k\ge 1}$ form an orthonormal system in $L^2(\mathcal{S})$.


\subsection{Tikhonov regularized estimator}\label{subsec:tikhonov}

Given that the inverse problem in \eqref{eq:dis_model} is typically ill-posed, we employ the Tikhonov regularization method to obtain a stable solution. Let us briefly recall the formulation of this classical approach. The Tikhonov estimator balances data fidelity with a regularity constraint by minimizing the following penalized least squares functional:
\begin{equation} \label{eq:loss_func}
    \mathcal{E}_\lambda(\phi) = \|\mcA\phi - \mathbf{b}\|_\spaceY^2 + \lambda^2 \|\phi\|_{\calC}^2.
\end{equation}
The data-fidelity term is the squared $\spaceY$-norm. Under the white-noise model of Section~\ref{subsec:problem_setting} this term is formally infinite; the minimization is then understood through its equivalent variational form
\begin{equation} \label{eq:loss_func_weak}
    \arg\min_\phi \mathcal{E}_\lambda(\phi) = \arg\min_\phi \Big\{ \langle \mcA^\top\mcA\,\phi, \phi\rangle_{L^2(\mathcal{S})} - 2\langle \mcA^\top\mathbf{b}, \phi\rangle_{L^2(\mathcal{S})} + \lambda^2 \|\phi\|_{\calC}^2 \Big\},
\end{equation}
in which every term is well defined since $\mcA^\top\mathbf{b}\in L^2(\mathcal{S})$ almost surely. The regularization term is defined via the inner product $\|\phi\|_{\calC}^2 = \langle \calC\phi, \phi \rangle_{L^2(\mathcal{S})}$, where $\calC: L^2(\mathcal{S}) \to L^2(\mathcal{S})$ is a self-adjoint, positive semi-definite operator. The parameter $\lambda > 0$ controls the trade-off between fitting the noisy data and enforcing the prior smoothness encoded by $\calC$.

The explicit solution to the minimization problem \eqref{eq:loss_func} is given by the normal equations, yielding the closed-form estimator:
\begin{equation} \label{eq:phi_Tikh}
    \hat{\phi}_\lambda = (\mcA^\top \mcA + \lambda^2 \calC)^{\dagger} \mcA^\top \bb.
\end{equation}
In this work, we consider the class of regularization operators
\begin{equation}\label{eq:Cs_family}
    \calC = (\mcA^\top \mcA)^{-s}, \qquad s \ge 0,
\end{equation}
where the exponent $s$ controls how strongly the penalty adapts to the operator spectrum, ranging from a non-adaptive penalty at $s=0$ to increasingly data-adaptive penalties as $s$ grows. Two choices of $s$ are of particular interest: $s=0$ (i.e., $\calC=I$), classical ridge regression (RR), and $s=1$ (i.e., $\calC=(\mcA^\top\mcA)^\dagger$), the standard form of data-adaptive RKHS regularization (DA)~\cite{LLA22,lu2024nonparametric}. The numerical experiments in Section~\ref{sec:numerics} are carried out with these two common regularization choices.

\subsection{From data roughness to operator spectral decay}

In the classical theory of regularized inverse problems, the eigenvalue decay rate of the normal operator $\mcA^\top \mcA$ is a key quantity that determines the degree of ill-posedness and, in turn, the achievable convergence rate of the estimator. A distinctive feature of the kernel learning problem \eqref{eq:dis_model} is that the operator $\mcA$ is constructed from the input data $\{u_m\}$, so its spectral properties are not fixed but depend on the input functions. This raises the central question: how do the properties of the input data impact the spectral decay of $\mcA^\top \mcA$, and with it the accuracy of kernel recovery?

We approach this question from two sides. First, on the operator side, we identify the eigenvalue decay rate $\gamma$ of $\mcA^\top \mcA$ as the key quantity controlling recovery accuracy. Second, on the data side, we parametrize input roughness by a spectral decay rate $\alpha$ and study its link to $\gamma$.

\textbf{Spectral characterization of operator ill-posedness.}
Consider the spectral decomposition of the normal operator:
\begin{equation}
\mcA^\top \mcA = \sum_{k \ge 1} \lambda_k \, \psi_k \otimes \psi_k,
\end{equation}
where $\{\lambda_k\}_{k\ge 1}$ is the non-increasing sequence of eigenvalues (converging to zero) and $\{\psi_k\}_{k\ge 1}$ is the corresponding orthonormal basis of eigenfunctions in $L^2(\mathcal{S})$.
We assume that the eigenvalue decay follows a polynomial rate.

\begin{assumption}[Operator spectral decay]\label{assump:operator_spectral_decay}
The eigenvalues of $\mcA^\top \mcA$ decay polynomially: there exists a decay rate parameter $\gamma > 0$ and constants $0 < \underline{c} \le \overline{c}$ such that:
\begin{equation*}
    \underline{c}\,k^{-\gamma} \le \lambda_k \le \overline{c}\,k^{-\gamma}, \quad k \ge 1.
\end{equation*}
Since our convergence analysis requires $\sum_k\lambda_k^{s+1}<\infty$, where $s$ is the regularization exponent of \eqref{eq:Cs_family}, we additionally assume $\gamma>1/(s+1)$.
\end{assumption}

A larger $\gamma$ corresponds to a more severely ill-posed problem.
Since $\mcA$ depends on the training data, $\gamma$ is data-dependent rather than an intrinsic operator property. To quantify this dependence, we need a measure of data roughness.

\textbf{Spectral characterization of data roughness.}
Data roughness can be quantified in various ways (e.g., Sobolev regularity or total variation), but in the present context the spectral decay rate of the input functions is the most natural choice for two reasons. 
First, the convergence-rate theory of Tikhonov regularization is conventionally formulated in the spectral domain of the forward operator~\cite{engl1996regularization,kirsch2021introduction,hansen2010discrete}, so the influence of the data on the operator eigenvalues is most directly captured in the frequency domain. 
Second, as we show below, this Fourier-based parameterization admits a clean analytical link from the data spectral decay rate to the operator spectral decay rate for a broad class of translation-invariant interaction kernels.

We adopt the following Fourier-based model for the input functions:
\begin{equation}\label{eq:data_u}
u(x):=\sum_{k=1}^{K} X_k k^{-\alpha} (\cos(c_{\delta} k x)+\sin(c_{\delta} k x)) , \quad K\in\mathbb{N}\cup\{+\infty\},
\end{equation}
where $\{X_k\}_{k=1}^{K}$ are independent standard normal random variables and $c_{\delta}=2\pi/\delta$, with $\delta$ being the support size of the kernel (i.e., $\mathcal{S} = [0,\delta]$).
The parameter $\alpha > 0$ controls the rate at which the Fourier coefficients decay: a smaller $\alpha$ corresponds to richer high-frequency content, which we refer to as rougher data. We note that this single-parameter model does not capture all forms of data complexity (e.g., discontinuities or localized features); extending the analysis beyond this spectral decay framework is left for future work.

\begin{remark}\label{rmk:data_model_generality}
The specific frequency choice $c_\delta = 2\pi/\delta$ is for analytical convenience: it aligns the Fourier modes with the kernel support and is what enables the closed-form $\alpha\mapsto\gamma$ relation in Lemma~\ref{lem:spectrum_general}. For data outside this Fourier-aligned structure, the closed-form is lost, but the operator spectral decay rate $\gamma$ can still be determined numerically by fitting the eigenvalues of $\mcA^\top \mcA$.
\end{remark}

\textbf{From data roughness to operator ill-posedness.}
How does the data roughness parameter $\alpha$ determine the operator spectral decay rate $\gamma$? Since $\mcA$ is built from the random training inputs $\{u_m\}_{m=1}^M$, the sample normal operator $\mcA^\top \mcA$ is itself random; a closed-form $\alpha\mapsto\gamma$ relation can therefore only be obtained in the large-sample limit $M\to\infty$. Under the averaged inner product of Section~\ref{subsec:problem_setting}, the empirical normal operator $\mcA^\top\mcA$ converges to its expectation
\begin{equation}\label{eq:L_def}
\mathcal{L}:L^2(\mathcal{S}) \to L^2(\mathcal{S}), \qquad
\langle \mathcal{L}\phi, \psi \rangle_{L^2(\mathcal{S})} := \mathbb{E}_u\!\bigl[\langle R_{\phi}[u],\, R_{\psi}[u] \rangle_{L^2(\Omega)}\bigr],
\end{equation}
as $M\to\infty$, with operator-norm sampling error of order $\mathcal{O}_p(M^{-1/2})$. Weyl's inequality then gives, for each fixed $k$,
\begin{equation}\label{eq:weyl_bridge}
\lambda_k(\mcA^\top\mcA) \;=\; \lambda_k(\mathcal{L}) + \mathcal{O}_p(M^{-1/2}),
\end{equation}
so whenever $\lambda_k(\mathcal{L}) \gg M^{-1/2}$, the eigenvalues of $\mcA^\top\mcA$ inherit both the polynomial decay rate and the spectral constants of $\mathcal{L}$ with no $M$-dependent rescaling. The bounds in Assumption~\ref{assump:operator_spectral_decay} are thus realized with the exponent $\gamma$ and the $M$-independent constants $\underline{c}, \overline{c}$ of $\mathcal{L}$ computed below. For a broad class of translation-invariant interaction terms, $\mathcal{L}$ admits a closed-form spectrum, as the following lemma shows.

\begin{lemma}[Spectrum of $\mathcal{L}$ under the shift-factorized form]\label{lem:spectrum_general}
Assume that
(i) the interaction term has the shift-factorized form
\begin{equation}\label{eq:G2-form}
g[u](x, s) = (\calB u)(x - s),
\end{equation}
where $\calB$ is a translation-invariant operator with Hermitian-symmetric Fourier symbol $\widehat{\calB}$ (see Appendix~\ref{app:lemma} for precise definitions);
(ii) the observation domain $\Omega = [0,L]$ satisfies $2L/\delta \in \mathbb{N}$.
Then the non-zero eigenvalues of $\mathcal{L}$ are
\begin{equation}\label{eq:lambda_general}
\lambda_j(\mathcal{L}) = \frac{L\delta}{2}\,|\widehat{\calB}(\theta_j)|^2\,j^{-2\alpha}, \qquad j=1,\dots,K,
\end{equation}
where $\theta_j := c_\delta\,j = 2\pi j/\delta$ is the $j$-th Fourier frequency of the data model \eqref{eq:data_u}.
If, in addition, $K=\infty$ and $\widehat{\calB}$ has polynomial decay of order $\rho$, i.e.,\ $c_1|k|^{-\rho}\le|\widehat{\calB}(k)|\le c_2|k|^{-\rho}$ as $|k|\to\infty$ for constants $0<c_1\le c_2$,
then the eigenvalues of $\mathcal{L}$ decay polynomially at rate $\gamma = 2(\alpha+\rho)$: $\underline c_{\mathcal L}\,j^{-\gamma}\le\lambda_j(\mathcal{L})\le\overline c_{\mathcal L}\,j^{-\gamma}$ for constants $0<\underline c_{\mathcal L}\le\overline c_{\mathcal L}$.
\end{lemma}

We defer the proof of this lemma to Appendix \ref{app:lemma}.
The lemma reduces the question ``how does $\alpha$ determine $\gamma$?'' 
to inspecting the Fourier multiplier of the fixed processing operator $\calB$. Table~\ref{tab:lemma_instances} lists three natural instances covered by Lemma~\ref{lem:spectrum_general}.

\begin{table}[h]
\centering
\small
\setlength{\tabcolsep}{4pt}
\caption{Three instances of Lemma~\ref{lem:spectrum_general}.}
\label{tab:lemma_instances}
\begin{tabular}{@{}llll@{}}
\toprule
Operator & $g[u](x,s)$ & $\widehat{\calB}(k)$ & $\gamma$ \\
\midrule
Convolution & $u(x-s)$ & $1$ & $2\alpha$ \\
Convolution with smooth kernel $G$ & $(G\!*\!u)(x-s)$ & $\widehat G(k)\asymp |k|^{-\rho}$ & $2(\alpha+\rho)$ \\
Fractional smoothing & $((-\Delta)^{-\rho/2}u)(x-s)$ & $|k|^{-\rho}$ & $2(\alpha+\rho)$ \\
\bottomrule
\end{tabular}
\end{table}

An important example outside the shift-factorized class is the nonlocal operator which has $g[u](x,s) = u(x-s)+u(x+s)-2u(x)$. It is translation-invariant, but since $u(x-s)+u(x+s)-2u(x) = (T_s+T_{-s}-2I)u(x)$ is a sum of three different shifts, the associated Fourier multiplier $m_s(k) = 2\cos(ks)-2$ depends on $s$, and $g$ cannot be written as $(\calB u)(x-s)$ for a single $s$-independent operator $\calB$. Lemma~\ref{lem:spectrum_general} therefore does not apply, and no closed-form $\alpha\mapsto\gamma$ map is available; the corresponding $\gamma$ has to be determined numerically by fitting the eigenvalues of $\mcA^\top \mcA$. More broadly, for interaction terms outside the shift-factorized form \eqref{eq:G2-form}, no closed-form $\alpha\mapsto\gamma$ relation is available, and $\gamma$ is again determined numerically.

\subsection{Effective regularized dimension}

In the error analysis of regularized estimators, the effective regularized dimension~\cite{zhang2005learning} naturally arises in the variance term of the error bound as a measure of the effective complexity of the hypothesis space after regularization. It counts the number of spectral directions of $\mcA^\top \mcA$ that are effectively retained at a given regularization level $\lambda$.

\begin{definition}[Effective regularized dimension]
The effective regularized dimension of the inverse problem under the Tikhonov penalty $\lambda\|\phi\|_{\calC}^2$ is defined as:
\begin{equation}\label{eq:eff_dim_def}
    d_{\rm eff}(\lambda) = \operatorname{tr}\left( \mcA(\mcA^\top \mcA + \lambda^2 \calC)^{\dagger} \mcA^\top \right).
\end{equation}
Moreover, assume that the operators $\mcA^\top \mcA$ and $\calC$ share a common orthonormal eigenbasis $\{\psi_k\}_{k=1}^\infty$ with corresponding eigenvalues:
\[
    \mcA^\top \mcA\,\psi_k = \lambda_k\,\psi_k, \quad \calC\,\psi_k = \mu_k\,\psi_k,
\]
for $k \geq 1$. The effective dimension can then be expressed spectrally as:
\begin{equation}\label{eq:eff_dim_sum}
    d_{\rm eff}(\lambda) = \sum_{k:\lambda_k>0} \frac{\lambda_k}{\lambda_k + \lambda^2 \mu_k}.
\end{equation}
\end{definition}

Substituting the regularization class $\calC=(\mcA^\top\mcA)^{-s}$ (so $\mu_k=\lambda_k^{-s}$) from Section~\ref{subsec:tikhonov} into \eqref{eq:eff_dim_sum} gives the explicit form
\begin{equation}\label{eq:eff_dim_general}
    d_{\rm eff}^{(s)}(\lambda) = \sum_{k:\lambda_k>0} \frac{\lambda_k^{s+1}}{\lambda_k^{s+1} + \lambda^2}.
\end{equation}




Intuitively, as the regularization vanishes ($\lambda \to 0$), the estimator relies entirely on the data, and $d_{\rm eff}(\lambda)$ approaches the rank of $\mcA$. Conversely, under strong regularization ($\lambda \to \infty$), the penalty term dominates, and the effective dimension approaches zero. Thus, $d_{\rm eff}(\lambda)$ continuously interpolates between the full complexity of the unregularized inverse problem and the trivial zero solution, reflecting the trade-off between data-fit and penalization.

\section{Convergence analysis in the small-noise limit} \label{sec:errbd}

In this section, we derive error bounds for the Tikhonov estimator $\hat{\phi}_\lambda$ defined in \eqref{eq:phi_Tikh}. Our analysis focuses on the convergence behavior in the small-noise limit. Throughout this section, $\|\cdot\|_2$ denotes the $L^2(\mathcal{S})$ norm on the kernel space.

The convergence rate is governed by the interplay between the smoothness of the true kernel $\phi_*$ and the ill-posedness of the inverse problem. 
Expanding the kernel in the eigenfunctions of $\mcA^\top\mcA$, we formalize this smoothness through the standard source condition from inverse problem theory.

\begin{definition}[Smoothness exponent and smoothness classes]\label{def:smoothness}
For a function $\phi$ expanded in the orthonormal eigenbasis $\{\psi_k\}_{k\ge1}$ of $\mcA^\top\mcA$,
\begin{equation}\label{eq:source_cond}
\phi = \sum_{k\ge1}c_k\,\psi_k ,\qquad c_k:=\langle\phi,\psi_k\rangle ,
\end{equation}
the \emph{smoothness exponent} of a class $\mathcal{F}$ of such functions is
\begin{equation}\label{eq:smoothness_def}
\beta:=\sup\Bigl\{\mu>0:\ \sum_{k\ge1}k^{2\mu}c_k^2<\infty\ \text{ for all }\ \phi\in\mathcal{F}\Bigr\},
\end{equation}
the largest exponent $\mu$ for which this weighted sum stays finite across the whole class. We consider the following two specific classes,
\begin{align}
\text{Sobolev-ball class:}&\quad \mathcal{F}_{\mathrm{Sob}}^{\tilde\beta}:=\Bigl\{\phi:\ \textstyle\sum_{k\ge1} k^{2\tilde\beta}c_k^2\le\mcE\Bigr\}, \label{eq:class_sob}\\
\text{power-law class:}&\quad \mathcal{F}_{\mathrm{pow}}^{\tilde\beta}:=\Bigl\{\phi:\ a_1 k^{-\tilde\beta}\le|c_k|\le a_2 k^{-\tilde\beta},\ \forall k\ge1\Bigr\}, \label{eq:class_pow}
\end{align}
with the rate parameter $\tilde\beta>0$ for $\mathcal{F}_{\mathrm{Sob}}^{\tilde\beta}$ and $\tilde\beta>\tfrac12$ for $\mathcal{F}_{\mathrm{pow}}^{\tilde\beta}$, and constants $\mcE,a_1,a_2>0$. $\mathcal{F}_{\mathrm{Sob}}^{\tilde\beta}$ bounds only the weighted sum, constraining the coefficients on average, whereas $\mathcal{F}_{\mathrm{pow}}^{\tilde\beta}$ constrains each coefficient to an exact two-sided rate.
\end{definition}

We next make $\beta$ explicit for these two classes.

\begin{lemma}\label{lem:beta_classes}
The classes \eqref{eq:class_sob} and \eqref{eq:class_pow} have smoothness exponents
\begin{equation}\label{eq:beta_eff}
\beta=\tilde\beta\ \ \text{for }\mathcal{F}_{\mathrm{Sob}}^{\tilde\beta},\qquad
\beta=\tilde\beta-\tfrac12\ \ \text{for }\mathcal{F}_{\mathrm{pow}}^{\tilde\beta} ,
\end{equation}
with the supremum attained on $\mathcal{F}_{\mathrm{Sob}}^{\tilde\beta}$ (i.e., it is a maximum), whereas on $\mathcal{F}_{\mathrm{pow}}^{\tilde\beta}$ it is not attained.
\end{lemma}

\begin{proof}
For the Sobolev-ball class, when $\mu\le\tilde\beta$ we have $k^{2\mu}\le k^{2\tilde\beta}$ for every $k\ge1$, so
\[
\sum_{k\ge1} k^{2\mu}c_k^2\;\le\;\sum_{k\ge1} k^{2\tilde\beta}c_k^2\;\le\;\mcE<\infty ,
\]
whereas for $\mu>\tilde\beta$ the sum can diverge; hence $\beta=\tilde\beta$, attained. For the power-law class, the upper bound $|c_k|\le a_2 k^{-\tilde\beta}$ gives
\[
\sum_{k\ge1} k^{2\mu}c_k^2\;\le\; a_2^2\sum_{k\ge1} k^{2\mu-2\tilde\beta},
\]
finite when $2\mu-2\tilde\beta<-1$, i.e.\ $\mu<\tilde\beta-\tfrac12$. Conversely, the saturating element $c_k=a_2 k^{-\tilde\beta}$ lies in $\mathcal{F}_{\mathrm{pow}}^{\tilde\beta}$ (since $a_2\ge a_1$) and gives $\sum_{k\ge1}k^{2\mu}c_k^2 = a_2^2\sum_{k\ge1}k^{2\mu-2\tilde\beta}$, which diverges for $\mu\ge\tilde\beta-\tfrac12$. Hence $\beta=\tilde\beta-\tfrac12$, with the supremum not achieved.
\end{proof}


\begin{assumption}[Source condition]\label{assump:source}
The true kernel lies in the eigenspace of $\mcA^\top\mcA$, i.e.\ $\phi_*\in\ker(\mcA^\top\mcA)^\perp$, so that it admits the expansion \eqref{eq:source_cond}. We further assume that $\phi_*$ belongs to one of the two smoothness classes of Definition~\ref{def:smoothness}, $\mathcal{F}_{\mathrm{Sob}}^{\tilde\beta}$ or $\mathcal{F}_{\mathrm{pow}}^{\tilde\beta}$.
\end{assumption}


With these assumptions in place, we now state the main result of this paper, which gives explicit convergence rates for the regularization class \eqref{eq:Cs_family}, depending on its exponent $s$, the operator spectral decay $\gamma$ and the kernel smoothness $\beta$.

\begin{theorem}[Small-noise convergence rates]\label{theorem:small_noise}
Under Assumptions \ref{assump:operator_spectral_decay} and \ref{assump:source} together with the white-noise model of Section~\ref{subsec:problem_setting}, with the regularization operator $\calC=(\mcA^\top\mcA)^{-s}$ ($s\ge 0$) and the optimal regularization parameter $\lambda = \lambda_*$, the kernel error of the Tikhonov estimator satisfies the following upper bound as $\sigma\to 0$:
\begin{equation}\label{eq:opti_lambda_rate}
\E\left[ \| \hat{\phi}_\lambda - \phi_{*} \|_2^2 \mid \mcA\right] \le
\begin{cases}
    C_1\,\sigma^{\frac{4\beta}{2\beta+\gamma+1}}, & \gamma>\beta/(s+1) \quad \bigl(\lambda_* \asymp \sigma^{\frac{(s+1)\gamma}{2\beta+\gamma+1}}\bigr), \\[4pt]
    C_2\,\sigma^{\frac{4(s+1)}{(2s+3)+1/\gamma}}, & \gamma<\beta/(s+1) \quad \bigl(\lambda_* \asymp \sigma^{\frac{s+1}{(2s+3)+1/\gamma}}\bigr).
\end{cases}
\end{equation}
The constants $C_1, C_2$ depend on $\gamma$, $\beta$, $s$, the spectral constants in Assumption \ref{assump:operator_spectral_decay}, the source-condition constants, and the sample size $M$.
\end{theorem}

The proof of this theorem is based on a bias-variance decomposition of the Tikhonov estimator in the eigenbasis of $\mcA^\top \mcA$ followed by optimization of $\lambda$, and is deferred to Appendix~\ref{app:proof_theorem}. The transition point $\gamma=\beta/(s+1)$ itself is excluded from \eqref{eq:opti_lambda_rate}: Appendix~\ref{app:proof_theorem} shows that for $\phi_*\in\mathcal{F}_{\mathrm{Sob}}^{\tilde\beta}$ the boundary rate equals the common limit of the two branches in \eqref{eq:opti_lambda_rate}, whereas for $\phi_*\in\mathcal{F}_{\mathrm{pow}}^{\tilde\beta}$ it picks up a singular extra multiplicative $|\log\sigma|$ factor exactly at the boundary that is absent for any $\gamma\ne\beta/(s+1)$.

The main qualitative feature of this result is a phase transition in the convergence rate, set by the interplay between the data roughness and the kernel smoothness.
Viewed as a function of the operator spectral decay $\gamma$, the convergence rate exponent, denoted $p(\gamma)$, is non-monotone and peaks at the threshold $\gamma^*=\beta/(s+1)$, reflecting a bias-variance trade-off; the right panel of Fig.~\ref{fig:rate_vs_gamma} visualizes this non-monotone behavior for the two common cases $s=0$ and $s=1$. We refer to $\gamma>\beta/(s+1)$ as the under-rough regime and $\gamma<\beta/(s+1)$ as the over-rough regime.
In the under-rough regime the regularization bias dominates, and roughening the data makes more high-frequency directions of the kernel observable, so $p(\gamma)$ increases. In the over-rough regime the bias has saturated and the rate becomes independent of $\beta$, so further roughening only inflates the variance and $p(\gamma)$ decreases.
Thus roughening the data helps only until its roughness matches the kernel's smoothness.

\begin{remark}[Role of the regularization exponent $s$]\label{rmk:role_s}
We remark on how the regularization exponent $s$ affects the convergence rate, now viewing the rate exponent as a function of $s$.
In the over-rough regime, a larger $s$ raises the convergence rate, since the exponent $4(s+1)/[(2s+3)+1/\gamma]$ grows with $s$. This gain is bounded, however: the peak exponent $p(s)$ increases monotonically with $s$ and, as $s\to\infty$, approaches the minimax-optimal value $4\beta/(2\beta+1)$. A larger $s$ also lowers the threshold $\gamma^*=\beta/(s+1)$, moving the peak toward smaller $\gamma$, so rougher data plays a larger role. In the under-rough regime, by contrast, the rate $\sigma^{4\beta/(2\beta+\gamma+1)}$ does not depend on $s$, so increasing $s$ brings no improvement there; we confirm this in our following numerical experiments.
\end{remark}

\section{Numerical experiments}\label{sec:numerics} 


The experiments are organized into four parts. We first validate the theory in an idealized diagonal matrix setting in Section~\ref{subsec:diagA}, where the spectral structure is fully controlled.
In Section~\ref{subsec:shared_eig}, the kernel shares the eigenbasis with the operator, and the theoretical assumptions in Theorem~\ref{theorem:small_noise} are satisfied. Section~\ref{subsec:general} treats a general case in which the kernel and data are constructed independently, so the operator and kernel need not share a common eigenbasis; this setting reflects many practical scenarios. Finally, Section~\ref{subsec:nn} extends the study to neural-network-based kernel regression, showing that the spectral decay of the data continues to influence recovery accuracy beyond our current theoretical framework. Together, these four experiments form a progression from idealized to realistic settings: Sections~\ref{subsec:diagA}--\ref{subsec:shared_eig} test Theorem~\ref{theorem:small_noise} quantitatively where its assumptions hold exactly, while Sections~\ref{subsec:general}--\ref{subsec:nn} probe the qualitative persistence of the predicted trends as the assumptions are progressively relaxed.

\textbf{Discretization.}
The theoretical analysis in Sections~\ref{sec:math_framework}--\ref{sec:errbd} is formulated for the continuous operator $\mcA: L^2(\mathcal{S}) \to \spaceY$ defined in \eqref{eq:A_def}, and Theorem~\ref{theorem:small_noise} bounds the squared continuous $L^2(\mathcal{S})$ error. The experiments are conducted on a uniform grid of spacing $h$: the kernel is discretized as $\bm\phi = (\phi_1,\dots,\phi_N) \in \mathbb{R}^N$ with $\phi_i = \phi(s_i)$ at the nodes $\{s_i\}_{i=1}^N\subset\mathcal{S}$, each input $u_m$ is sampled at the grid nodes, and the outputs are evaluated at the $J$ nodes $\{x_j\}_{j=1}^J\subset\Omega$. Approximating the integral in \eqref{eq:A_def} by the midpoint rule yields the discrete forward map
\begin{equation*}
  (A\bm\phi)_{(m,j)} \;=\; h\sum_{i=1}^N \phi_i\,g[u_m](x_j,s_i),
  \qquad A\in\mathbb{R}^{MJ\times N}.
\end{equation*}
The discrete Tikhonov estimator $\hat{\bm\phi}\in\mathbb{R}^N$ is the finite-dimensional analogue of \eqref{eq:phi_Tikh}, and the experiments report the squared discrete $L^2$ error
\begin{equation*}
  \|\hat{\bm\phi} - \bm\phi_*\|^2_h \;=\; h \sum_{i=1}^N (\hat\phi_i - \phi_{*,i})^2,
\end{equation*}
where $\bm\phi_* = (\phi_*(s_1),\dots,\phi_*(s_N))\in\mathbb{R}^N$ is the ground-truth kernel sampled on the grid. Here $\|\cdot\|_h$ denotes the discrete $L^2$ norm.

\textbf{Oracle-optimal regularization.}
We use the oracle-optimal $\lambda_*$, computed by evaluating $\hat{\bm\phi}_\lambda$ on a fine logarithmic grid of $\lambda$ and selecting the value that minimizes the squared discrete $L^2$ error $\|\hat{\bm\phi}_\lambda - \bm\phi_*\|^2_h$. We use the oracle choice because the small-noise convergence rate is a fine-grained property: a per-noise-level deviation in $\lambda$ that is negligible for any single error value can compound into a substantial bias in the fitted slope, so the oracle isolates the rate from $\lambda$-selection noise and provides a clean benchmark for Theorem~\ref{theorem:small_noise}.

\subsection{Verification of rate phase transition on diagonal matrices}\label{subsec:diagA}

We validate Theorem~\ref{theorem:small_noise} in a controlled diagonal setting where $A^\top A$ has eigenvalues $\lambda_k = k^{-\gamma}$ with standard orthonormal eigenvectors, fully decoupling the kernel from the data spectrum. In this construction, $\gamma$ directly controls the operator spectral decay as in Assumption~\ref{assump:operator_spectral_decay}, and $n$ denotes the truncation level: the number of retained spectral modes, so that $A^\top A$ is an $n\times n$ diagonal matrix with eigenvalues $\{\lambda_k\}_{k=1}^{n}$.

We compare the convergence exponents predicted by Theorem~\ref{theorem:small_noise} against empirical rates obtained by fitting $\log(\text{error})=p\cdot\log(\sigma)+c$ across noise levels, using the oracle-optimal $\lambda_*$ defined above. To complete this controlled construction, the ground-truth kernel is prescribed directly through its coefficients in the eigenbasis $\{\psi_k\}_{k=1}^n$ (rather than inherited from a data-driven model): we set the eigen-coefficients to $c_{k,*}:=\langle\phi_*,\psi_k\rangle = k^{-\tilde\beta}$, placing $\phi_*$ in the power-law class $\mathcal{F}_{\mathrm{pow}}^{\tilde\beta}$ of Definition~\ref{def:smoothness}. By Lemma~\ref{lem:beta_classes}, the source-condition decay rate is $\beta=\tilde\beta-\tfrac12$.

Fig.~\ref{fig:rate_vs_alpha_combined} plots the convergence rate $p$ as a function of the operator spectral decay rate $\gamma$ for $\beta=2.5,3.5$ (corresponding to $\tilde\beta=3,4$). The solid curves are the rates from Theorem~\ref{theorem:small_noise}, obtained by substituting the source-condition decay rate $\beta$ into \eqref{eq:opti_lambda_rate}, and the markers are empirical rates fitted from oracle-optimal inverse-problem solves at $n=2500$ spectral modes with $50$ Monte Carlo replicates, over the five noise levels $\sigma\in\{10^{-5},3\!\times\!10^{-5},10^{-4},3\!\times\!10^{-4},10^{-3}\}$. The decay rate $\gamma$ is swept over a grid in $[1,5]$; the RR markers omit $\gamma=1$, the RR validity boundary $\gamma=1/(s+1)=1$ of Assumption~\ref{assump:operator_spectral_decay}, where the effective dimension diverges. Shaded bands show $\pm 1$ standard deviation across the per-replicate slope fits; they quantify the run-to-run variability of the fitted exponent, which grows at larger $\gamma$ as the increasingly ill-posed problem makes the fit more sensitive to the noise realization. The mean fitted rates nonetheless stay on the theoretical curves across the full range of $\gamma$, with the phase-transition peaks at the predicted thresholds $\gamma^*_{\RR}=\beta$ and $\gamma^*_{\DA}=\beta/2$, i.e.,~$(\gamma^*_{\RR},\gamma^*_{\DA}) = (2.5,1.25)$ for $\beta=2.5$ and $(3.5,1.75)$ for $\beta=3.5$.
Having confirmed these rates and phase-transition thresholds in the idealized diagonal setting, the following sections test whether the agreement persists in more realistic operator configurations.

\begin{figure}[htbp]
\centering
\includegraphics[width=0.95\textwidth]{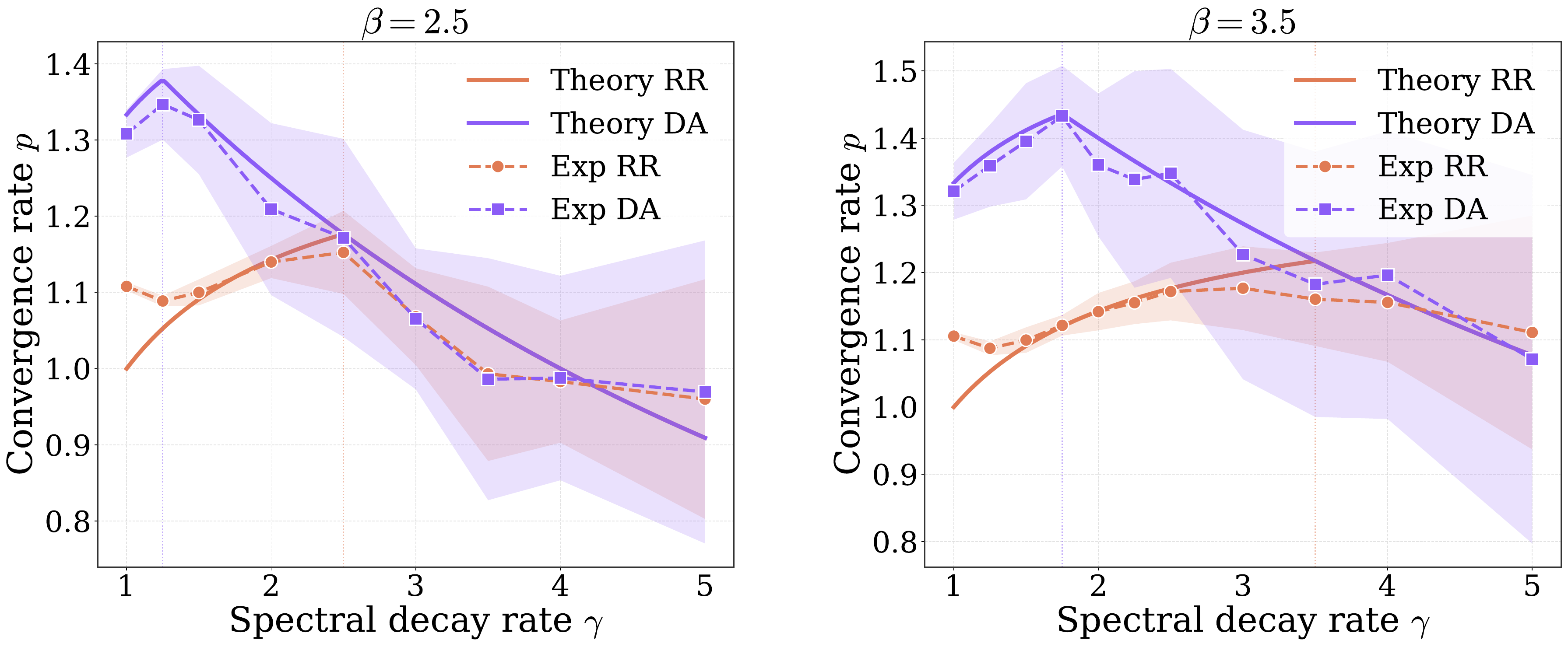}
\caption{Diagonal matrix example: convergence rate $p$ vs.\ operator spectral decay rate $\gamma$. Empirical rates closely follow the predictions of Theorem~\ref{theorem:small_noise}.}
\label{fig:rate_vs_alpha_combined}
\vspace{-1em}
\end{figure}

\subsection{Shared eigenvector case}\label{subsec:shared_eig}

This subsection tests Theorem~\ref{theorem:small_noise} on two representative integral operators whose target kernel shares the operator's eigenbasis, varying the data roughness $\alpha$ to traverse both the over-rough and under-rough branches of the predicted convergence rate.

\textbf{Operators.} We consider two representative operator instances of the model \eqref{eq:model}: the convolution operator
\begin{equation}\label{eq:op_linear}
R_\phi[u](x) = \int_{0}^{\delta} \phi(s)\, u(x-s)\,ds, \quad x\in[0,1],
\end{equation}
and the nonlocal operator
\begin{equation}\label{eq:op_nonlocal}
R_\phi[u](x) = \int_{0}^{\delta} \phi(s)\big(u(x-s)+u(x+s)-2u(x)\big)\,ds, \quad x\in[0,1].
\end{equation}
These two operators serve as a complementary pair with respect to the analytical framework of Lemma~\ref{lem:spectrum_general}: the convolution operator is shift-factorized with $\calB=I$ (so $\widehat{\calB}\equiv 1$ and $\rho=0$), and Lemma~\ref{lem:spectrum_general} predicts $\gamma=2\alpha$; the nonlocal operator is translation-invariant but not shift-factorized, so it falls outside the hypotheses of Lemma~\ref{lem:spectrum_general} and no closed-form $\gamma$ is available, and we determine $\gamma$ for it by numerically fitting the eigenvalues of $A^\top A$. We fix the kernel support size to $\delta = 0.25$ (so that $\Omega = [0,1]$ and $2/\delta = 8 \in \mathbb N$ satisfy the condition of Lemma~\ref{lem:spectrum_general}). The same two operators are also used in Sections~\ref{subsec:general}--\ref{subsec:nn}.

The two operators differ in how the kernel reconstruction error must be reported. Theorem~\ref{theorem:small_noise} is stated under the source-condition Assumption~\ref{assump:source}, which requires $\bm\phi_*\in\ker(A)^\perp$ (equivalently $P_{\ker(A)}\bm\phi_*=0$), and the Tikhonov estimator satisfies $\hat{\bm\phi}\in\mathrm{range}(A^\top)=\ker(A)^\perp$. The squared error therefore admits the orthogonal decomposition
\begin{equation*}
\|\hat{\bm\phi}-\bm\phi_*\|^2_h \;=\; \|P_{\ker(A)^\perp}(\hat{\bm\phi}-\bm\phi_*)\|^2_h \;+\; \|P_{\ker(A)}\bm\phi_*\|^2_h,
\end{equation*}
in which only the first term is governed by the rate $O(\sigma^p)$ predicted by the theorem. For the convolution operator the second term is at the level of numerical residue, so we plot $\|\hat{\bm\phi}-\bm\phi_*\|^2_h$ directly. For the nonlocal operator the second term is structurally non-negligible because the kernel argument $u(x-s)+u(x+s)-2u(x)$ is invariant under $s\to-s$, so the operator is blind to the odd-in-$s$ part of the discretized kernel and $A$ is rank-deficient; to remain within the scope of the theorem we therefore plot the in-range error $\|P_{\ker(A)^\perp}(\hat{\bm\phi}-\bm\phi_*)\|^2_h = \|\hat{\bm\phi}-\bm\phi_*\|^2_h - \|P_{\ker(A)}\bm\phi_*\|^2_h$ for the nonlocal operator.

\textbf{Setup.} Unlike the diagonal setting where $\gamma$ is prescribed independently, here $\gamma$ is genuinely determined by the input data roughness $\alpha$, and the kernel shares the Fourier eigenbasis with the operator so that the source-condition Assumption~\ref{assump:source} holds. Following the same parameterization as Section~\ref{subsec:diagA}, the target-kernel eigen-coefficients are $c_{k,*}=k^{-\tilde\beta}$, placing $\phi_*$ in the power-law class $\mathcal{F}_{\mathrm{pow}}^{\tilde\beta}$ with $\beta=\tilde\beta-\tfrac12$. The input $u$ is drawn from the spectral model \eqref{eq:data_u} with finite truncation $K<\infty$, followed by $L^\infty$-normalization
\vspace{-0.3em}
\begin{equation}\label{eq:data_u_exp}
u(x):=\tilde u(x)/\|\tilde u\|_{L^\infty([0,1])},
\vspace{-0.3em}
\end{equation}
which keeps the effective noise level approximately constant across $\alpha$ values. Noisy outputs are generated as $\bb \leftarrow A\phib_* + \mathrm{nsr}\cdot\sigma_b\cdot\bm\xi$ with $\bm\xi \sim \mathcal{N}(0,I)$, where $\sigma_b$ denotes the maximum absolute value of the noiseless output samples over all data indices $(m,j)$; in contrast to the diagonal setting of Section~\ref{subsec:diagA}, which prescribes the per-observation noise standard deviation $\sigma$ directly, here the output magnitude depends on the input data and must be normalized by $\sigma_b$, so the per-observation $\sigma = \mathrm{nsr}\cdot\sigma_b$ and sweeping $\mathrm{nsr}$ is equivalent to sweeping $\sigma$. The same noise model is used in Sections~\ref{subsec:general}--\ref{subsec:nn}. We report results at $\tilde\beta=4$, $M=100$ training pairs, and grid spacing $h=6.25\times10^{-4}$. The fitted $\gamma$ values span $\gamma\in[1.4, 9.5]$ across $\alpha\in[0.75, 5]$, so the experiment traverses both the over-rough branch ($\gamma<\beta/(s+1)$, small $\alpha$) and the under-rough branch ($\gamma>\beta/(s+1)$, large $\alpha$) of Theorem~\ref{theorem:small_noise}.

\textbf{Results.} Fig.~\ref{fig:shared_eig_error_vs_noise} plots the kernel error against the noise level $\sigma$. For both operators the fitted slopes agree well with the convergence rates of Theorem~\ref{theorem:small_noise}, and DA is consistently somewhat more accurate than RR across the tested range.
Beyond the convergence rate, the quantity we care about most is the accuracy of the recovered kernel, so we next look at how the error itself varies with the data roughness $\alpha$. Fig.~\ref{fig:shared_eig_error_vs_alpha} plots the error against $\alpha$ at fixed noise. The error first decreases with $\alpha$, reaches a minimum, then increases again, a phase transition that is clear for both convolution and nonlocal operators. The roughness that gives the smallest error, however, does not coincide with the one that gives the fastest rate. Theorem~\ref{theorem:small_noise} places the fastest rate at $\gamma^*_{\RR}=\beta$ and $\gamma^*_{\DA}=\beta/2$, that is $\alpha^*_{\RR}\approx1.75$ and $\alpha^*_{\DA}\approx0.875$ through $\gamma\approx2\alpha$, whereas the convolution error is smallest near $\alpha\approx3$ for RR and $\alpha\approx1.75$ for DA.
This gap is what one should expect at a finite noise level. The error scales like a $\gamma$-dependent prefactor times $\sigma^{p(\gamma)}$, where the exponent $p(\gamma)$ is the convergence rate. The threshold $\gamma^*$ maximizes only the exponent, but the prefactor also changes with $\gamma$, and at a nonzero $\sigma$ it is large enough to move the smallest-error roughness to a larger value. Only as $\sigma\to0$, where the $\sigma^{p(\gamma)}$ factor takes over, does the minimum return to $\gamma^*$.

In short, this shared-eigenbasis test confirms Theorem~\ref{theorem:small_noise} on two counts: the measured convergence rates match the predicted exponents for both operators, and the kernel error itself shows the predicted phase transition in the data roughness. We next ask how much of this behavior survives once the shared-eigenbasis assumption is dropped.

\begin{figure}[htbp]
    \centering
    \includegraphics[width=\textwidth]{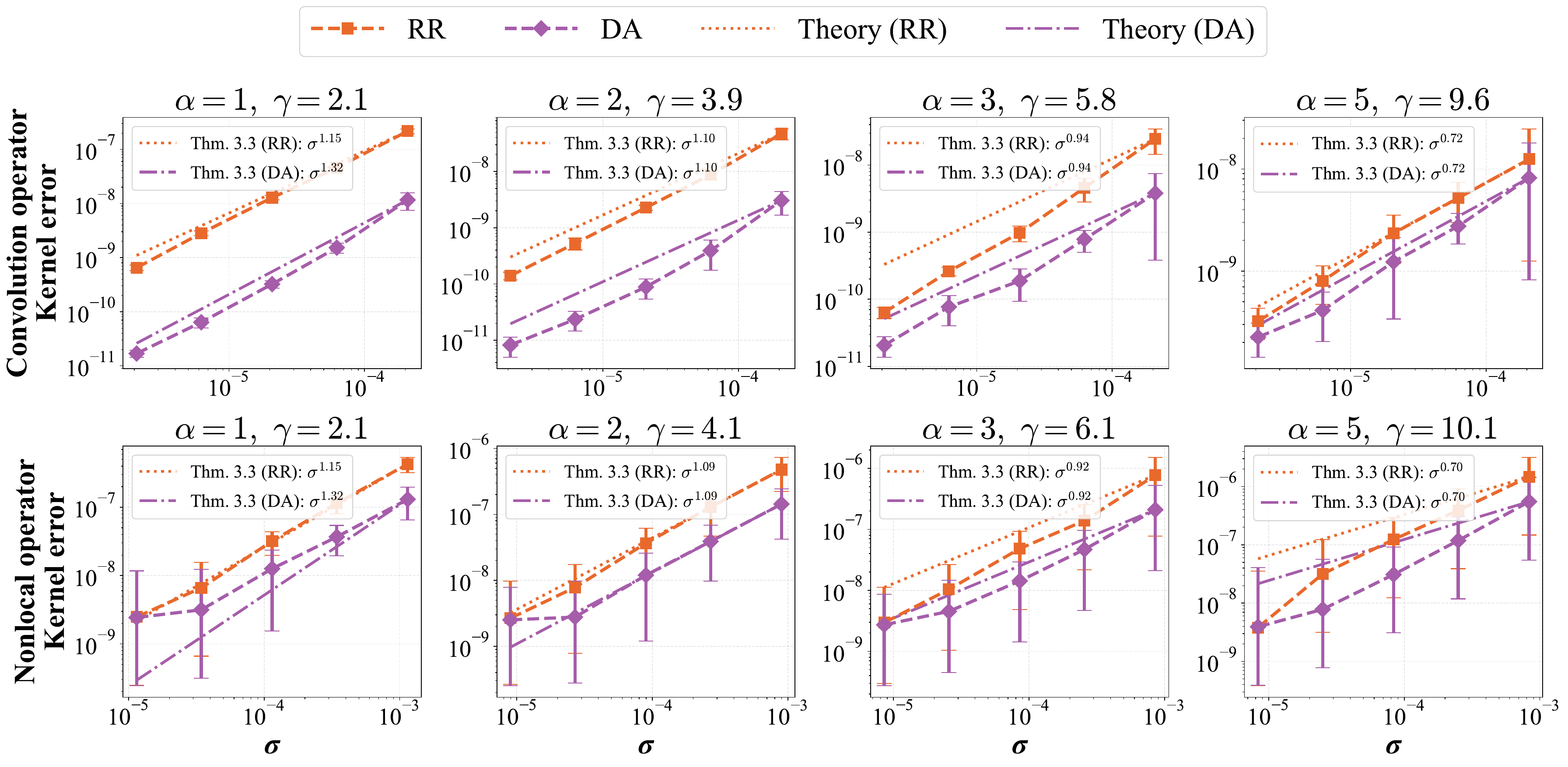}
    \vspace{-3mm} \caption{Shared eigenvector example ($\tilde\beta=4$): kernel error vs.\ noise $\sigma$. Empirically fitted RR and DA convergence rates closely match the predictions of Theorem~\ref{theorem:small_noise}.}
    \label{fig:shared_eig_error_vs_noise}
    \vspace{-2em}
\end{figure}

\begin{figure}[htbp]
    \centering
    \includegraphics[width=\textwidth]{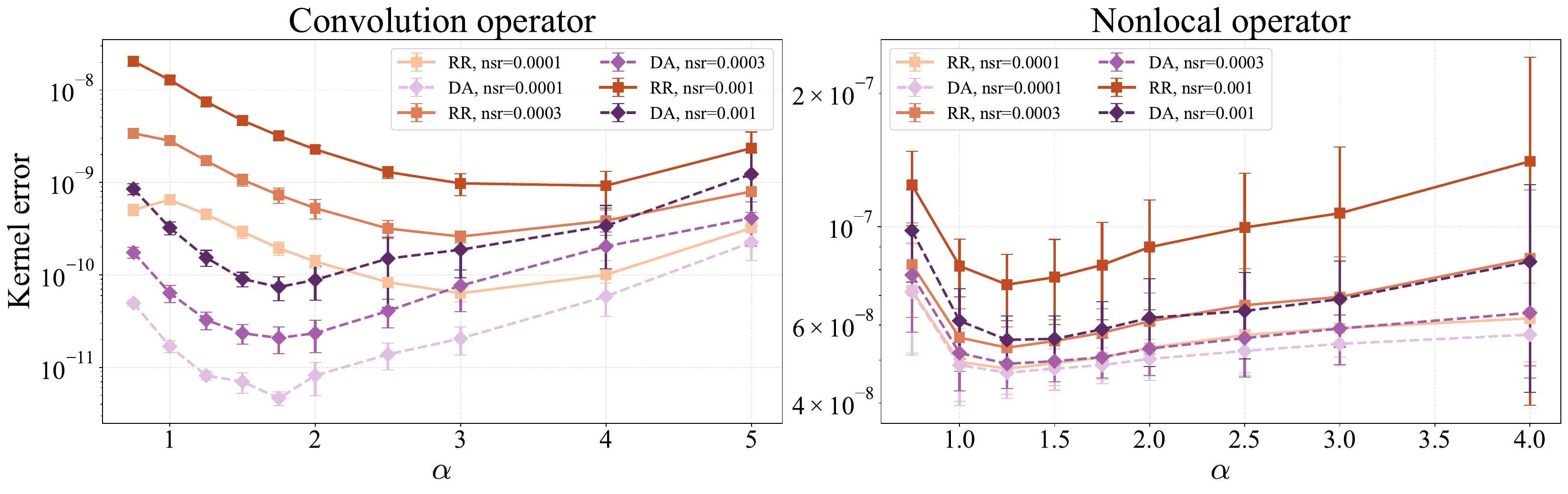}
    \vspace{-3mm} \caption{Shared eigenvector example ($\tilde\beta=4$): kernel error vs.\ data roughness $\alpha$. The error exhibits a clear V-shape phase transition.}
    \label{fig:shared_eig_error_vs_alpha}
    \vspace{-2em}
\end{figure}


\subsection{General case}\label{subsec:general}
We now relax the most restrictive assumption, namely that the target kernel and the normal operator $\mcA^\top\mcA$ share a common eigenbasis, and ask how much of this picture survives.

\textbf{Setup.} The kernel and the data are constructed independently, so $A^\top A$ and the target kernel do not share a common eigenbasis. The kernel is expanded in a fixed Fourier basis as $\phi_*(s)=\sum_{k=1}^{K_\phi} c_k\cos(2\pi k s/\delta)$ with coefficients $c_k = X_k\,k^{-\tilde\beta}$, where $X_k\overset{\mathrm{iid}}{\sim}\mathcal N(0,1)$ are independent standard Gaussian draws, while $u$ is drawn from the same spectral model \eqref{eq:data_u_exp} with decay $\alpha$. We use $\tilde\beta=4$, $K_\phi=100$, and the same operators \eqref{eq:op_linear}--\eqref{eq:op_nonlocal}. The source-condition Assumption~\ref{assump:source} is now violated: $\phi_*$ does not live in the eigenbasis of $A^\top A$, so Theorem~\ref{theorem:small_noise} does not formally apply. We therefore use the theorem only as a reference point.

\textbf{Results.} In the error-versus-noise view (Fig.~\ref{fig:general_error_vs_noise}), for every tested $\alpha$ and for both RR and DA, the error decreases monotonically as $\sigma\to 0$, so the small-noise convergence behavior is preserved; the fitted $\gamma$ values are approximately $2\alpha$ for both operators (consistent with Section~\ref{subsec:shared_eig}), and DA slopes remain steeper than RR in the under-rough regime. Figure~\ref{fig:general_error_vs_alpha} plots the error against $\alpha$ at a fixed noise level $\mathrm{nsr}=5\times10^{-4}$, and a clear V shape no longer appears. What remains is a rough but consistent trend over the tested range: the error grows with $\alpha$, so rougher data (smaller $\alpha$) gives a more accurate kernel. The trend is strongest for the nonlocal operator, where the error climbs by about two orders of magnitude from the roughest to the smoothest data, and is also clear for the convolution operator (most so for DA); only the convolution RR curve stays close to flat. A clean V is hard to resolve here for two reasons: the kernel $\phi_*$ no longer lies in the eigenbasis of $A^\top A$, so the source condition assumed by Theorem~\ref{theorem:small_noise} does not hold and the phase transition need not appear sharply, and the scatter from the random kernel draws and the limited sampling may further hide any weak minimum. We therefore do not expect quantitative agreement with the theory once its assumptions are dropped; what does come through plainly is the strong effect of data roughness on kernel recovery. 

\begin{figure}[htbp]
    \centering
    \includegraphics[width=\textwidth]{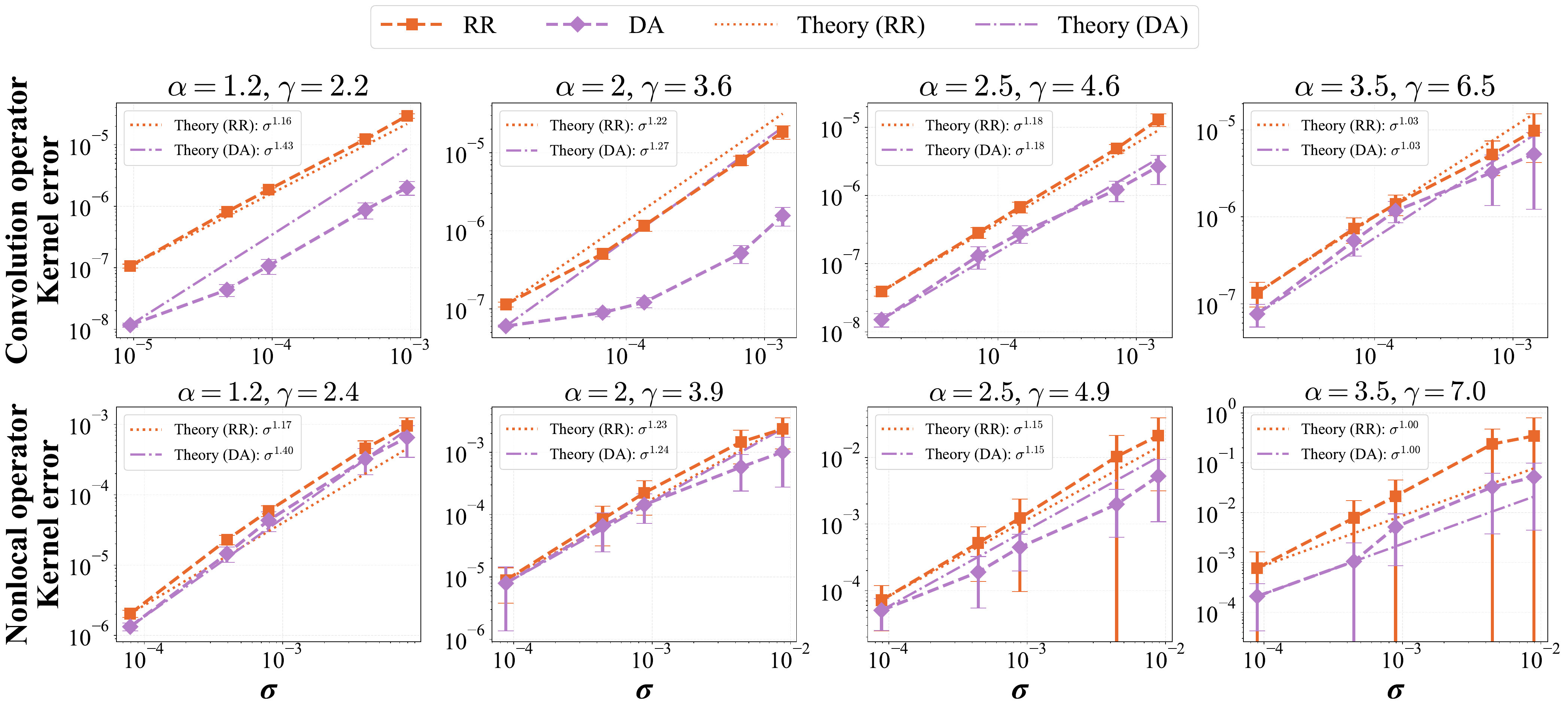}
    \vspace{-2em}\caption{General case ($\tilde\beta=4$): kernel error vs.\ noise $\sigma$. Small-noise convergence persists even without a shared eigenbasis, with DA slopes steeper than RR in the under-rough regime.}
    \label{fig:general_error_vs_noise}
    \vspace{-2em}
\end{figure}

\begin{figure}[htbp]
    \centering
    \includegraphics[width=0.85\textwidth]{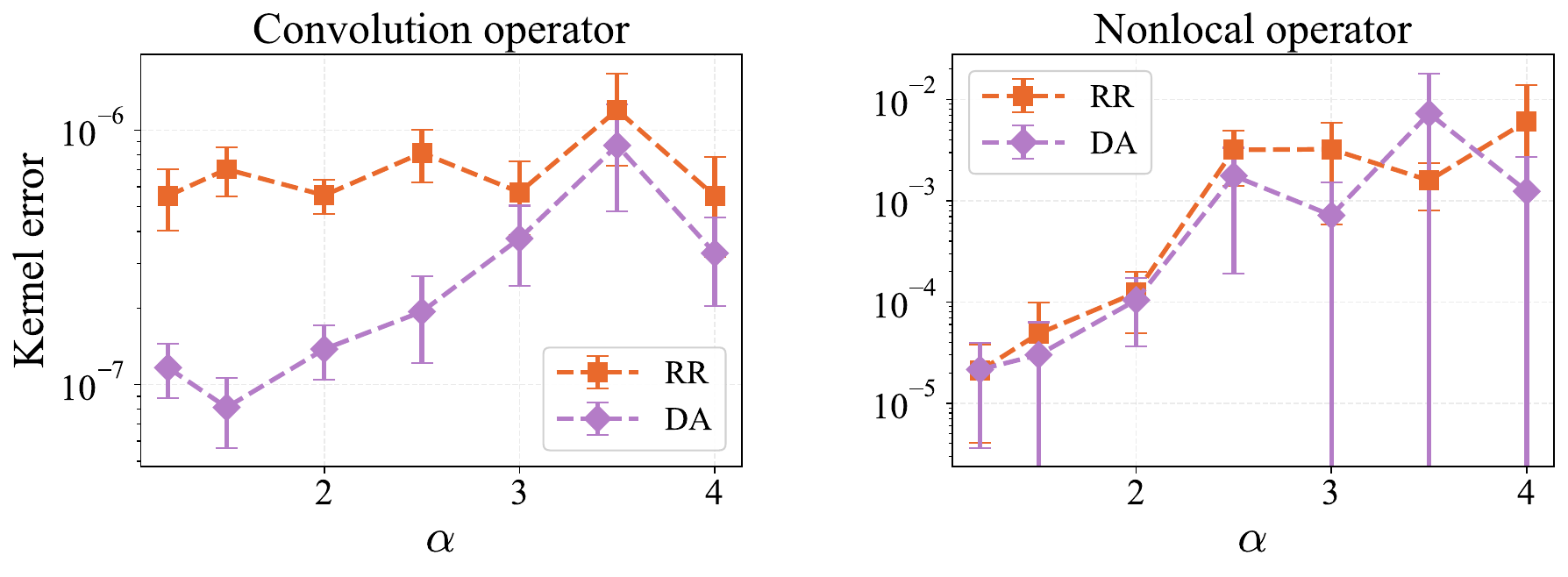} 
    \vspace{-1em}\caption{General case ($\tilde\beta=4$): kernel error vs.\ data roughness $\alpha$. The V-shape phase transition in $\alpha$ qualitatively carries over to the non-aligned setting, and DA attains uniformly lower error than RR across the tested range.}
    \label{fig:general_error_vs_alpha}
    \vspace{-2em}
\end{figure}


\subsection{Extension: Neural network kernel regression}\label{subsec:nn}
We extend the investigation to neural-network-based kernel regression, which lies outside the Tikhonov framework underlying Theorem~\ref{theorem:small_noise}. The goal here is exploratory: to check whether the qualitative effect of data roughness on kernel recovery persists when the estimator is replaced by a neural network, without attempting a quantitative match to the theoretical rates.

\textbf{Training setup.} The kernel $\phi$ is parameterized by a fully-connected feed-forward neural network $\phi_\theta: [0,\delta]\to\mathbb R$ with four hidden layers of width 128 and ReLU activations (layer widths $[1,128,128,128,128,1]$), embedded in a graph kernel network~\cite{you2021data} that evaluates $R_{\phi_\theta}[u](x)$ by a Riemann-sum discretization of the integral. The training objective is the mean-squared output residual
\begin{equation*}
    \mathcal J(\theta) \;=\; \frac{1}{MJ}\sum_{m=1}^M\sum_{j=1}^J\bigl(R_{\phi_\theta}[u_m](x_j) - b_{mj}\bigr)^2
\end{equation*}
over the training pairs $\{(u_m,b_m)\}$. Input functions $u_m$ are drawn from \eqref{eq:data_u_exp} with spectral decay $\alpha$; noisy outputs $b_m$ are generated via the two operators. We use $M=100$ training samples, with $30$ additional validation samples and $20$ test samples. Optimization uses the Adam optimizer with initial learning rate $\mathrm{lr}\in\{10^{-3}, 5\times 10^{-3}\}$ (we report the minimum error over this small grid), a two-phase exponential decay schedule ($\mathrm{lr}_t = \mathrm{lr}_0 \cdot 0.995^t$ for the first $30\%$ of training and $\mathrm{lr}_0 \cdot 0.995^{1500}\cdot 0.998^{\,t-1500}$ thereafter), zero weight decay, and batch size $10$. Networks are trained for $5000$ epochs with default PyTorch initialization. All training configurations use $\tilde\beta = 2$.

\textbf{Results.} Fig.~\ref{fig:nn} shows the kernel reconstruction error as a function of data roughness $\alpha$ at three representative noise-to-signal ratios $\mathrm{nsr}\in\{0.05, 0.1, 0.5\}$, for both operators. The pattern is qualitatively consistent across the two operators and all three noise levels: rougher input data (smaller $\alpha$) yields lower kernel error, while very smooth data leads to a sharp accuracy loss. Since the neural-network estimator lies outside the Tikhonov framework of Theorem~\ref{theorem:small_noise}, we do not attempt a quantitative comparison with the theoretical phase-transition thresholds. These results nonetheless suggest that the data-roughness effect is not an artifact of the regularized least-squares framework: it appears to reflect a more fundamental property of the kernel learning problem, and data roughness remains a practically relevant design variable regardless of the estimator used.

\begin{figure}[htbp]
  \centering
    \includegraphics[width=0.4\textwidth]{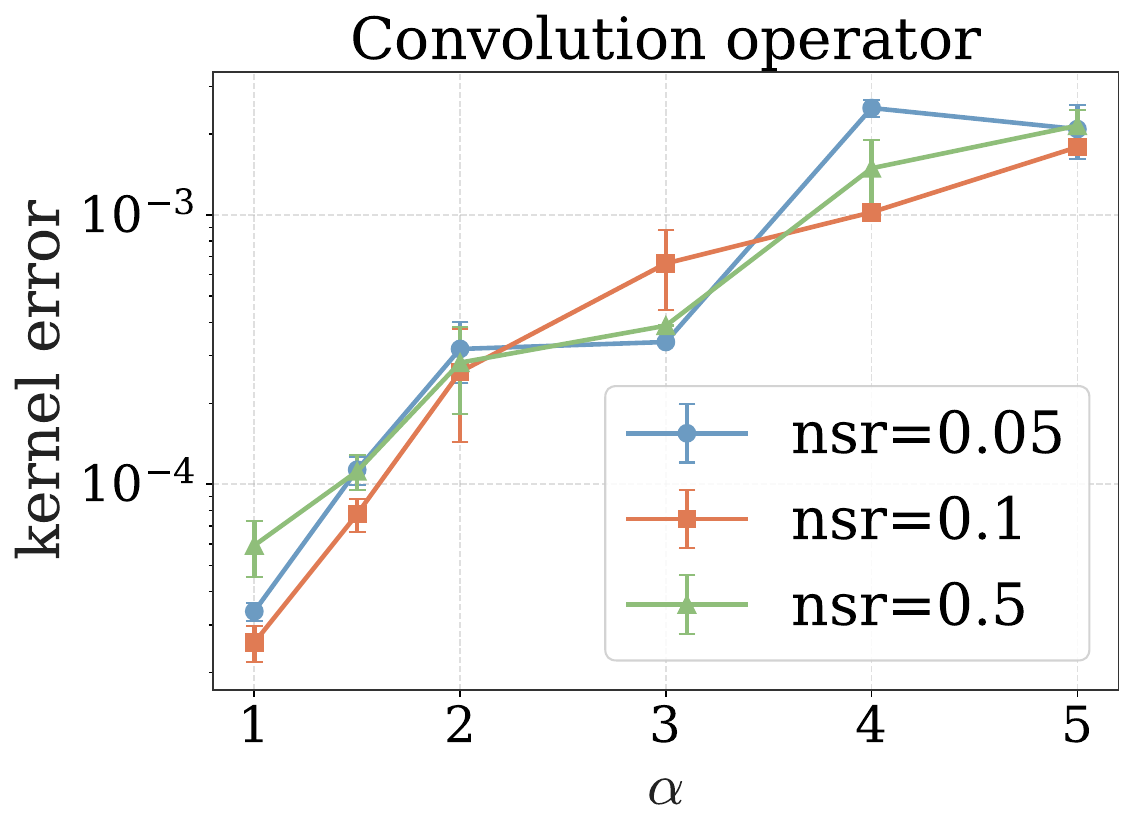}
    \includegraphics[width=0.4\textwidth]{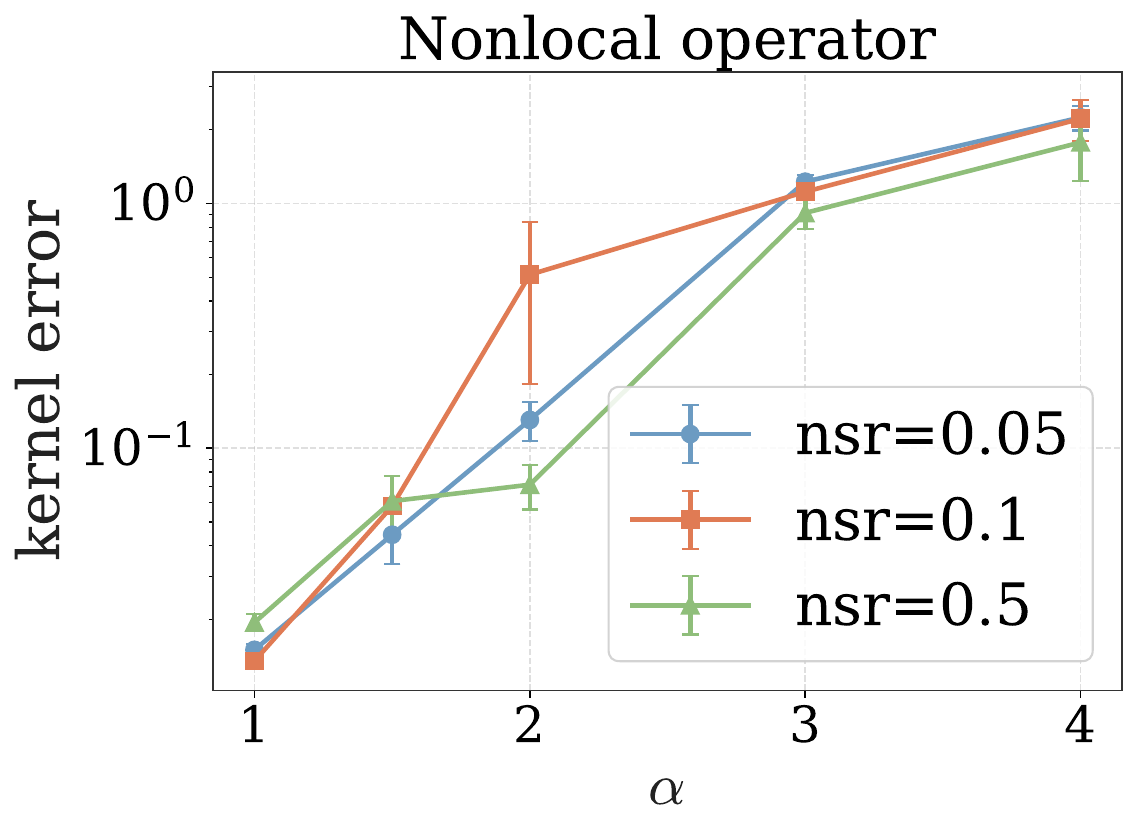}
  \vspace{-3mm} \caption{Neural network kernel regression: kernel error vs.\ data roughness $\alpha$.}
  \label{fig:nn}
    \vspace{-2em}
\end{figure}

\section{Discussion and conclusion}\label{sec:conclusion} 

This paper investigates how the roughness of input data affects kernel recovery in integral operators. Under polynomial eigenvalue decay of the forward operator and a source condition on the target kernel, the convergence rate of the Tikhonov estimator exhibits a phase transition (Theorem~\ref{theorem:small_noise}): in the under-rough regime, spectrally richer (rougher) data improves recovery, while in the over-rough regime beyond the transition threshold further roughening becomes detrimental. A larger regularization exponent $s$ delays the over-rough regime and sharpens the rate within it, while leaving the under-rough rate unchanged. Numerical experiments across multiple problem settings confirm these predictions.


Several questions remain open. First, Lemma~\ref{lem:spectrum_general} links 
the input data's spectral decay rate $\alpha$ and the operator spectral decay $\gamma$ for shift-factorized interactions with translation-invariant Fourier symbol; extending this connection beyond that class (e.g., non-translation-invariant or non-factorized interactions) would broaden the framework. Second, our analysis assumes the operator and the kernel share a common eigenbasis; extending the theory to the non-aligned case would broaden its applicability. Third, the beneficial effect of spectrally rich data is observed empirically in neural network kernel regression beyond our current theoretical scope, and developing theory for that setting is an important direction for future work.

\bibliographystyle{siamplain}
\bibliography{yyu}

\begin{thebibliography}{10}

\bibitem{askari2008peridynamics}
{\sc E.~Askari, F.~Bobaru, R.~Lehoucq, M.~Parks, S.~Silling, and O.~Weckner}, {\em Peridynamics for multiscale materials modeling}, in Journal of Physics: Conference Series, vol.~125, IOP Publishing, 2008, p.~012078.

\bibitem{bissantz2007convergence}
{\sc N.~Bissantz, T.~Hohage, A.~Munk, and F.~Ruymgaart}, {\em Convergence rates of general regularization methods for statistical inverse problems and applications}, SIAM Journal on Numerical Analysis, 45 (2007), pp.~2610--2636.

\bibitem{blanchard2018optimal}
{\sc G.~Blanchard and N.~M{\"u}cke}, {\em Optimal rates for regularization of statistical inverse learning problems}, Foundations of Computational Mathematics, 18 (2018), pp.~971--1013.

\bibitem{bongini2017inferring}
{\sc M.~Bongini, M.~Fornasier, M.~Hansen, and M.~Maggioni}, {\em Inferring interaction rules from observations of evolutive systems {I}: The variational approach}, Mathematical Models and Methods in Applied Sciences, 27 (2017), pp.~909--951.

\bibitem{boulle2023elliptic}
{\sc N.~Boulle, D.~Halikias, and A.~Townsend}, {\em Elliptic {PDE} learning is provably data-efficient}, Proceedings of the National Academy of Sciences, 120 (2023), p.~e2303904120.

\bibitem{burkovska2022optimization}
{\sc O.~Burkovska, C.~Glusa, and M.~D'elia}, {\em An optimization-based approach to parameter learning for fractional type nonlocal models}, Computers \& Mathematics with Applications, 116 (2022), pp.~229--244.

\bibitem{cai2012minimax}
{\sc T.~T. Cai and M.~Yuan}, {\em Minimax and adaptive prediction for functional linear regression}, Journal of the American Statistical Association, 107 (2012), pp.~1201--1216.

\bibitem{CD07}
{\sc A.~Caponnetto and E.~De~Vito}, {\em Optimal rates for the regularized least-squares algorithm}, Foundations of Computational Mathematics, 7 (2007), pp.~331--368.

\bibitem{cavalier2008nonparametric}
{\sc L.~Cavalier}, {\em Nonparametric statistical inverse problems}, Inverse Problems, 24 (2008), p.~034004.

\bibitem{crambes2009smoothing}
{\sc C.~Crambes, A.~Kneip, and P.~Sarda}, {\em Smoothing splines estimators for functional linear regression}, The Annals of Statistics, 37 (2009), pp.~35--72.

\bibitem{dehoop2023convergence}
{\sc M.~V. de~Hoop, N.~B. Kovachki, N.~H. Nelsen, and A.~M. Stuart}, {\em Convergence rates for learning linear operators from noisy data}, SIAM/ASA Journal on Uncertainty Quantification, 11 (2023), pp.~480--513.

\bibitem{donatelli2025basic}
{\sc J.~Donatelli, J.~Jakeman, M.~Shields, A.~Gelb, F.~Herrmann, S.~Jantre, J.~Larson, J.~Mueller, A.~Oberai, C.~Petra, et~al.}, {\em Basic research needs for inverse methods for complex systems under uncertainty}, tech. report, US Department of Energy (USDOE), Washington, DC (United States). Office of~…, 2025.

\bibitem{du2018peridynamic}
{\sc Q.~Du, Y.~Tao, and X.~Tian}, {\em A peridynamic model of fracture mechanics with bond-breaking}, Journal of Elasticity, 132 (2018), pp.~197--218.

\bibitem{du2020mathematics}
{\sc Q.~Du and X.~Tian}, {\em Mathematics of smoothed particle hydrodynamics: A study via nonlocal stokes equations}, Foundations of Computational Mathematics, 20 (2020), pp.~801--826.

\bibitem{d2017nonlocal}
{\sc M.~D’Elia, Q.~Du, M.~Gunzburger, and R.~Lehoucq}, {\em Nonlocal convection-diffusion problems on bounded domains and finite-range jump processes}, Computational Methods in Applied Mathematics, 17 (2017), pp.~707--722.

\bibitem{engl1996regularization}
{\sc H.~W. Engl, M.~Hanke, and A.~Neubauer}, {\em Regularization of inverse problems}, vol.~375, Springer, 1996.

\bibitem{geng2025end}
{\sc Y.~Geng, O.~Burkovska, L.~Ju, G.~Zhang, and M.~Gunzburger}, {\em An end-to-end deep learning method for solving nonlocal allen--cahn and cahn--hilliard phase-field models}, Computer Methods in Applied Mechanics and Engineering, 436 (2025), p.~117721.

\bibitem{geng2025parallel}
{\sc Y.~Geng, J.~Yin, E.~C. Cyr, G.~Zhang, and L.~Ju}, {\em Parallel-in-time solution of allen-cahn equations by integrating operator learning into the parareal method}, arXiv preprint arXiv:2510.07672,  (2025).

\bibitem{glusa2021fast}
{\sc C.~Glusa, H.~Antil, M.~D'Elia, B.~van Bloemen~Waanders, and C.~J. Weiss}, {\em A fast solver for the fractional helmholtz equation}, SIAM Journal on Scientific Computing, 43 (2021), pp.~A1362--A1388.

\bibitem{guo2024ib}
{\sc L.~Guo, H.~Wu, Y.~Wang, W.~Zhou, and T.~Zhou}, {\em Ib-uq: Information bottleneck based uncertainty quantification for neural function regression and neural operator learning}, Journal of Computational Physics, 510 (2024), p.~113089.

\bibitem{hall2007methodology}
{\sc P.~Hall and J.~L. Horowitz}, {\em Methodology and convergence rates for functional linear regression}, The Annals of Statistics, 35 (2007), pp.~70--91.

\bibitem{hansen2010discrete}
{\sc P.~C. Hansen}, {\em {Discrete inverse problems: Insight and algorithms}}, Society for Industrial and Applied Mathematics, 2010.

\bibitem{huang2022unified}
{\sc Y.~Huang, Q.~Li, R.~Li, F.~Zeng, and L.~Guo}, {\em A unified fast memory-saving time-stepping method for fractional operators and its applications}, Numerical Mathematics: Theory, Methods and Applications, 15 (2022), pp.~679--714.

\bibitem{kirsch2021introduction}
{\sc A.~Kirsch}, {\em An introduction to the mathematical theory of inverse problems}, vol.~120, Springer, 2021.

\bibitem{knapik2011bayesian}
{\sc B.~T. Knapik, A.~W. van~der Vaart, and J.~H. van Zanten}, {\em Bayesian inverse problems with {G}aussian priors}, The Annals of Statistics, 39 (2011), pp.~2626--2657.

\bibitem{kovachki2023neural}
{\sc N.~B. Kovachki, Z.~Li, K.~Azizzadenesheli, B.~Liu, K.~Bhattacharya, A.~M. Stuart, and A.~Anandkumar}, {\em Neural operator: {L}earning maps between function spaces with applications to {PDE}s}, Journal of Machine Learning Research, 24 (2023), pp.~1--97.

\bibitem{lang2022learning}
{\sc Q.~Lang and F.~Lu}, {\em Learning interaction kernels in mean-field equations of first-order systems of interacting particles}, SIAM Journal on Scientific Computing, 44 (2022), pp.~A260--A285.

\bibitem{lang2023small}
{\sc Q.~Lang and F.~Lu}, {\em Small noise analysis for {Tikhonov} and {RKHS} regularizations}, arXiv preprint arXiv:2305.11055,  (2023).

\bibitem{lanthaler2022error}
{\sc S.~Lanthaler, S.~Mishra, and G.~E. Karniadakis}, {\em Error estimates for {DeepONets}: A deep learning framework in infinite dimensions}, Transactions of Mathematics and Its Applications, 6 (2022), pp.~1--141.

\bibitem{li2024preconditioned}
{\sc H.~Li}, {\em A preconditioned krylov subspace method for linear inverse problems with general-form tikhonov regularization}, SIAM Journal on Scientific Computing, 46 (2024), pp.~A2607--A2633.

\bibitem{lu2024nonparametric}
{\sc F.~Lu, Q.~An, and Y.~Yu}, {\em Nonparametric learning of kernels in nonlocal operators}, Journal of Peridynamics and Nonlocal Modeling, 6 (2024), pp.~347--370.

\bibitem{LLA22}
{\sc F.~Lu, Q.~Lang, and Q.~An}, {\em Data adaptive {RKHS} {T}ikhonov regularization for learning kernels in operators}, Proceedings of Mathematical and Scientific Machine Learning, 190 (2022), pp.~158--172.

\bibitem{LZMT19}
{\sc F.~Lu, M.~Zhong, S.~Tang, and M.~Maggioni}, {\em Nonparametric inference of interaction laws in systems of agents from trajectory data}, Proceedings of the National Academy of Sciences, 116 (2019), pp.~14424--14433.

\bibitem{lu2021stochastic}
{\sc F.~Lu, M.~Zhong, S.~Tang, and M.~Maggioni}, {\em Learning interaction kernels in stochastic systems of interacting particles from multiple trajectories}, Foundations of Computational Mathematics, 22 (2022), pp.~1013--1067.

\bibitem{lu2021learning}
{\sc L.~Lu, P.~Jin, G.~Pang, Z.~Zhang, and G.~E. Karniadakis}, {\em Learning nonlinear operators via {DeepONet} based on the universal approximation theorem of operators}, Nature Machine Intelligence, 3 (2021), pp.~218--229.

\bibitem{mengesha2013analysis}
{\sc T.~Mengesha and Q.~Du}, {\em Analysis of a scalar nonlocal peridynamic model with a sign changing kernel}, Discrete and Continuous Dynamical Systems-B, 18 (2013), pp.~1415--1437.

\bibitem{pillonetto2010new}
{\sc G.~Pillonetto and G.~De~Nicolao}, {\em A new kernel-based approach for linear system identification}, Automatica, 46 (2010), pp.~81--93.

\bibitem{silling2000reformulation}
{\sc S.~A. Silling}, {\em Reformulation of elasticity theory for discontinuities and long-range forces}, Journal of the Mechanics and Physics of Solids, 48 (2000), pp.~175--209.

\bibitem{suzuki2023fractional}
{\sc J.~L. Suzuki, M.~Gulian, M.~Zayernouri, and M.~D’Elia}, {\em Fractional modeling in action: A survey of nonlocal models for subsurface transport, turbulent flows, and anomalous materials}, Journal of Peridynamics and Nonlocal modeling, 5 (2023), pp.~392--459.

\bibitem{tang2024identifiability}
{\sc S.~Tang, M.~Tuerkoen, and H.~Zhou}, {\em On the identifiability of nonlocal interaction kernels in first-order systems of interacting particles on riemannian manifolds}, SIAM Journal on Applied Mathematics, 84 (2024), pp.~2067--2086.

\bibitem{tikhonov1963solution}
{\sc A.~N. Tikhonov}, {\em Solution of incorrectly formulated problems and the regularization method}, Sov Dok, 4 (1963), pp.~1035--1038.

\bibitem{wang2026monotone}
{\sc J.~Wang, X.~Tian, Z.~Zhang, S.~Silling, S.~Jafarzadeh, and Y.~Yu}, {\em Monotone peridynamic neural operator for nonlinear material modeling with conditionally unique solutions}, Computer Methods in Applied Mechanics and Engineering, 453 (2026), p.~118792.

\bibitem{xu2021machine}
{\sc X.~Xu, M.~D'Elia, and J.~T. Foster}, {\em A machine-learning framework for peridynamic material models with physical constraints}, Computer Methods in Applied Mechanics and Engineering, 386 (2021), p.~114029.

\bibitem{xu2022machine}
{\sc X.~Xu, M.~D'Elia, C.~Glusa, and J.~T. Foster}, {\em Machine-learning of nonlocal kernels for anomalous subsurface transport from breakthrough curves}, arXiv preprint arXiv:2201.11146,  (2022).

\bibitem{you2021data}
{\sc H.~You, Y.~Yu, N.~Trask, M.~Gulian, and M.~D’Elia}, {\em Data-driven learning of nonlocal physics from high-fidelity synthetic data}, Computer Methods in Applied Mechanics and Engineering, 374 (2021), p.~113553.

\bibitem{yuan2010reproducing}
{\sc M.~Yuan and T.~T. Cai}, {\em A reproducing kernel hilbert space approach to functional linear regression}, The Annals of Statistics, 38 (2010), pp.~3412--3444.

\bibitem{zhang2025minimax}
{\sc S.~Zhang, X.~Wang, and F.~Lu}, {\em Minimax rates for learning kernels in operators}, arXiv preprint arXiv:2502.20368,  (2025).

\bibitem{zhang2005learning}
{\sc T.~Zhang}, {\em Learning bounds for kernel regression using effective data dimensionality}, Neural Computation, 17 (2005), pp.~2077--2098.

\end{thebibliography}

\appendix

\section{Proof of Lemma~\ref{lem:spectrum_general}}\label{app:lemma}

Before turning to the proof, we make precise the notions of translation-invariant operator, Fourier symbol, and the Hermitian-symmetric assumption that appear in the statement of Lemma~\ref{lem:spectrum_general}, and fix the boundary setting on $\Omega$ used throughout the calculation.

A bounded linear operator $\calB$ on $L^2(\mathbb{R})$ is translation-invariant if it commutes with translations $T_s v(x) := v(x-s)$. By the convolution theorem, any such operator acts as multiplication in the Fourier domain,
\[
\widehat{\calB v}(\xi) \;=\; \widehat{\calB}(\xi)\,\widehat{v}(\xi),\qquad v\in L^2(\mathbb{R}),
\]
where the function $\widehat{\calB}:\mathbb{R}\to\mathbb{C}$ is called the Fourier symbol of $\calB$. Typical examples are the identity ($\widehat{\calB}\equiv 1$), differentiation ($\widehat{\calB}(\xi)=i\xi$), convolution with a kernel $G$ ($\widehat{\calB}=\widehat G$), and the fractional Laplacian $(-\partial_x^2)^{\rho/2}$ ($\widehat{\calB}(\xi)=|\xi|^\rho$, see, e.g., \cite{glusa2021fast}). The Hermitian-symmetry assumption $\widehat{\calB}(-\xi) = \overline{\widehat{\calB}(\xi)}$ of Lemma~\ref{lem:spectrum_general} is equivalent to $\calB$ mapping real-valued functions to real-valued functions; at the discrete frequencies $\pm\theta_j$ appearing below it takes the form $\widehat{\calB}(-\theta_j) = \overline{\widehat{\calB}(\theta_j)}$, which we use repeatedly to recombine complex conjugate pairs into real quantities.

We fix the observation domain $\Omega = [0,L]$ with $L \in (\delta/2)\,\mathbb{N}$ throughout the proof. This choice yields $2\theta_j L = 4\pi j L/\delta \in 2\pi\mathbb{Z}$ for every $j\in\mathbb N$, ensuring that all cross-frequency boundary terms encountered below vanish exactly. The experimental setting $\Omega = [0,1]$ with $2/\delta \in \mathbb{N}$ satisfies this condition.

With these preliminaries in place, we now prove the lemma in two steps: first we compute the integral kernel of $\mathcal{L}$, then we diagonalize $\mathcal{L}$ in the Fourier basis on $[0,\delta]$. 
By the bilinear form \eqref{eq:L_def}, $\mathcal{L}$ has integral kernel
\[
\kappa(s,t) = \mathbb{E}_u\!\!\int_\Omega g[u](x,s)\,g[u](x,t)\,dx.
\]
Substituting the shift-factorized form \eqref{eq:G2-form} and the data model \eqref{eq:data_u},
\[
g[u](x,s) = (\calB u)(x-s),\qquad \calB u = \sum_{j=1}^K X_j\,j^{-\alpha}\,\calB v_j,\quad v_j(x):=\cos(\theta_j x)+\sin(\theta_j x).
\]
Writing $v_j$ in complex form $v_j(x) = \tfrac{1-i}{2}e^{i\theta_j x} + \tfrac{1+i}{2}e^{-i\theta_j x}$ and applying $\calB$ as a Fourier multiplier with symbol $\widehat{\calB}$ (using $\widehat{\calB}(-\theta_j) = \overline{\widehat{\calB}(\theta_j)}$),
\[
\calB v_j(x) = \tfrac{1-i}{2}\,\widehat{\calB}(\theta_j)\,e^{i\theta_j x} + \tfrac{1+i}{2}\,\overline{\widehat{\calB}(\theta_j)}\,e^{-i\theta_j x}.
\]
Since $\mathbb{E}[X_j X_{j'}] = \delta_{jj'}$,
\[
\kappa(s,t) = \sum_{j=1}^K j^{-2\alpha}\!\!\int_\Omega (\calB v_j)(x-s)\,(\calB v_j)(x-t)\,dx.
\]
Expanding the product yields four monomials in $\{e^{\pm i\theta_j x}\}$. The oscillatory monomials $e^{\pm 2i\theta_j x}$ integrate to zero over $\Omega$ since $2\theta_j L \in 2\pi\mathbb{Z}$; the two surviving $x$-independent monomials combine via $|(1-i)/2|^2 = 1/2$ and $\widehat{\calB}(\theta_j)\,\overline{\widehat{\calB}(\theta_j)} = |\widehat{\calB}(\theta_j)|^2$ to give
\begin{align*}
\kappa(s,t) &= L\sum_{j=1}^K j^{-2\alpha}\,|\widehat{\calB}(\theta_j)|^2\,\cos\bigl(\theta_j(s-t)\bigr) \\
&= L\sum_{j\ge 1} j^{-2\alpha}|\widehat{\calB}(\theta_j)|^2\,\bigl[\cos\theta_j s \cos\theta_j t + \sin\theta_j s \sin\theta_j t\bigr]. 
\end{align*}
The Fourier modes 
$\{\sqrt{1/\delta}, \sqrt{2/\delta}\cos\theta_j s,\,\sqrt{2/\delta}\sin\theta_j s\}_{j\ge 1}$ 
form a complete orthonormal basis of $L^2[0,\delta]$. Denote them by $\psi_{j,m}$ with $\psi_{j,1}= \sqrt{2/\delta}\cos\theta_j s$ and $\psi_{j,2}= \sqrt{2/\delta}\sin\theta_j s$. A direct computation gives, for $m\in\{1,2\}$ and $j\ge 1$,
\[
\mathcal{L}\psi_{j,m}(s) = \int_0^\delta \kappa(s,t)\,\psi_{j,m}(t)\,dt = \frac{L\delta}{2}\,|\widehat{\calB}(\theta_j)|^2\,j^{-2\alpha}\,\psi_{j,m}(s),
\]
which is \eqref{eq:lambda_general}. Moreover, the constant function $\psi_0\equiv\sqrt{1/\delta}$ lies in $\ker\mathcal{L}$, since for each $j\ge 1$, $\int_0^\delta\cos(\theta_j(s-t))\,dt=0$ by $\theta_j\delta=2\pi j$; hence the basis above exhausts the spectrum of $\mathcal{L}$. The decay rate $\gamma = 2(\alpha + \rho)$ then follows by combining \eqref{eq:lambda_general} with the two-sided bound $c_1 j^{-2\rho}\le|\widehat{\calB}(\theta_j)|^2\le c_2 j^{-2\rho}$ (from $|\theta_j|=2\pi j/\delta$ and the rate-$\rho$ decay of $\widehat{\calB}$), which gives $\underline c_{\mathcal L}\,j^{-\gamma}\le\lambda_j(\mathcal L)\le\overline c_{\mathcal L}\,j^{-\gamma}$. Each $\lambda_j(\mathcal{L})$ is doubly degenerate (eigenspace spanned by $\psi_{j,1},\psi_{j,2}$), so the eigenvalues counted with multiplicity satisfy $\lambda_k\asymp k^{-\gamma}$ (with constants enlarged by a factor of $2^\gamma$), matching Assumption~\ref{assump:operator_spectral_decay}. \qed

\section{The proof of Theorem \ref{theorem:small_noise}}\label{app:proof_theorem}

Before giving the proof of the Theorem \ref{theorem:small_noise}, we first provide an estimate of the effective dimension and a unified kernel error bound for the Tikhonov estimator.

\begin{corollary}[Bounds on the effective dimension]\label{cor:eff_dim_esti}
Let $d_{\rm eff}^{(s)}(\lambda)$ denote the effective dimension of the penalty $\calC=(\mcA^\top\mcA)^{-s}$ ($s\ge0$), defined in \eqref{eq:eff_dim_general}. Under Assumption~\ref{assump:operator_spectral_decay}, for all $0<\lambda<1$,
\begin{equation}\label{eq:esti_dim}
     C_1(\underline{c},\gamma)\, \lambda^{-\frac{2}{(s+1)\gamma}} -\tfrac{1}{2}
     \;\le\; d_{\rm eff}^{(s)}(\lambda) \;\le\;
     C_2^{(s)}(\overline{c},\gamma)\, \lambda^{-\frac{2}{(s+1)\gamma}} + 1,
\end{equation}
where $C_1(\underline{c},\gamma) = \underline{c}^{1/\gamma}/{2}$ and $C_2^{(s)}(\overline{c},\gamma)=\overline{c}^{1/\gamma}\bigl(1+\frac{1}{(s+1)\gamma-1}\bigr)$. The two common cases are ridge regression ($s=0$, $d_{\rm eff}^{RR}=d_{\rm eff}^{(0)}$) and data-adaptive RKHS ($s=1$, $d_{\rm eff}^{DA}=d_{\rm eff}^{(1)}$).
\end{corollary}

\begin{proof}
Since $\calC=(\mcA^\top\mcA)^{-s}$ is a power of $\mcA^\top\mcA$, the two operators share the eigenbasis $\{\psi_k\}$; from $\mcA^\top\mcA\,\psi_k=\lambda_k\psi_k$ it then follows that $\calC\psi_k=\lambda_k^{-s}\psi_k$, so
\[
(\mcA^\top\mcA+\lambda^2\calC)\psi_k
=\bigl(\lambda_k+\lambda^2\lambda_k^{-s}\bigr)\psi_k
=\lambda_k^{-s}\bigl(\lambda_k^{s+1}+\lambda^2\bigr)\psi_k .
\]
Hence $\mcA(\mcA^\top\mcA+\lambda^2\calC)^\dagger\mcA^\top$ has eigenvalues
$\lambda_k\cdot\dfrac{\lambda_k^{s}}{\lambda_k^{s+1}+\lambda^2}=\dfrac{\lambda_k^{s+1}}{\lambda_k^{s+1}+\lambda^2}$, and
\begin{equation}\label{eq:deff_spectral_sum}
d_{\rm eff}^{(s)}(\lambda)=\sum_{k\ge1}\frac{\lambda_k^{s+1}}{\lambda_k^{s+1}+\lambda^2}.
\end{equation}

Set $\mu_k := \lambda_k^{s+1}$ and $g:=(s+1)\gamma$. Assumption~\ref{assump:operator_spectral_decay} gives
\begin{equation}\label{eq:mu_two_sided}
\underline{c}^{\,s+1}\,k^{-g} \;\le\; \mu_k \;\le\; \overline{c}^{\,s+1}\,k^{-g},
\qquad g=(s+1)\gamma>1 ,
\end{equation}
where $g>1$ because $\gamma>\frac{1}{s+1}$ in Assumption~\ref{assump:operator_spectral_decay}. For the counting function $m(\lambda):=\#\{k\ge1:\mu_k\ge\lambda^2\}$, \eqref{eq:mu_two_sided} yields
\begin{equation}\label{eq:counting_bound}
\Bigl(\tfrac{\underline{c}^{\,s+1}}{\lambda^2}\Bigr)^{1/g}-1
\;\le\; m(\lambda) \;\le\;
\Bigl(\tfrac{\overline{c}^{\,s+1}}{\lambda^2}\Bigr)^{1/g}+1 =: K(\lambda).
\end{equation}

For $x\ge0$,
\[
\tfrac12\,\mathbf 1_{\{x\ge1\}}\;\le\;\frac{x}{x+1}\;\le\;\mathbf 1_{\{x\ge1\}}+x\,\mathbf 1_{\{x<1\}} .
\]
Applying this with $x=\mu_k/\lambda^2$ in \eqref{eq:deff_spectral_sum} and summing over $k$,
\begin{equation}\label{eq:deff_sandwich}
\tfrac12\,m(\lambda) \;\le\; d_{\rm eff}^{(s)}(\lambda) \;\le\; m(\lambda) + \frac{1}{\lambda^2}\sum_{k:\,\mu_k<\lambda^2}\mu_k .
\end{equation}

For the lower bound, using the left inequality of \eqref{eq:deff_sandwich}, the left inequality of \eqref{eq:counting_bound}, and $(\underline{c}^{\,s+1})^{1/g}=\underline{c}^{1/\gamma}$,
\[
d_{\rm eff}^{(s)}(\lambda)
\;\ge\;\tfrac12\,m(\lambda)
\;\ge\;\tfrac12\Bigl(\bigl(\tfrac{\underline{c}^{\,s+1}}{\lambda^2}\bigr)^{1/g}-1\Bigr)
\;=\;\frac{\underline{c}^{1/\gamma}}{2}\,\lambda^{-2/g}-\tfrac12 ,
\]
which is the lower bound in \eqref{eq:esti_dim} with $C_1(\underline c,\gamma)=\underline c^{1/\gamma}/2$.

For the upper bound, since $\{k:\mu_k<\lambda^2\}\subset\{k>K(\lambda)\}$ and $g>1$, the integral test gives $\sum_{k>K}k^{-g}\le K^{1-g}/(g-1)$; with $K(\lambda)\ge(\overline{c}^{\,s+1}/\lambda^2)^{1/g}$ and $1+\tfrac{1-g}{g}=\tfrac1g$,
\[
\frac{1}{\lambda^2}\sum_{k:\,\mu_k<\lambda^2}\mu_k
\;\le\;\frac{\overline{c}^{\,s+1}}{(g-1)\lambda^2}\,K(\lambda)^{1-g}
\;\le\;\frac{\overline{c}^{1/\gamma}}{g-1}\,\lambda^{-2/g} .
\]
Combining this with the right inequality of \eqref{eq:deff_sandwich} and $m(\lambda)\le K(\lambda)=\overline{c}^{1/\gamma}\lambda^{-2/g}+1$ from \eqref{eq:counting_bound},
\[
d_{\rm eff}^{(s)}(\lambda)
\;\le\;\overline{c}^{1/\gamma}\Bigl(1+\frac{1}{g-1}\Bigr)\lambda^{-2/g}+1
\;=\;C_2^{(s)}(\overline c,\gamma)\,\lambda^{-\frac{2}{(s+1)\gamma}}+1 ,
\]
where $C_2^{(s)}(\overline c,\gamma)=\overline c^{1/\gamma}\bigl(1+\tfrac{1}{(s+1)\gamma-1}\bigr)$ and $g=(s+1)\gamma$. The lower and upper bounds together give \eqref{eq:esti_dim}.
\end{proof}

The proof of Theorem~\ref{theorem:small_noise} bounds the bias and variance of the Tikhonov estimator separately in the eigenbasis of $\mcA^\top\mcA$ and then optimizes $\lambda$ to balance them. We first record the exact bias-variance decomposition of the mean-squared error (MSE) that serves as its starting point. 

\begin{lemma}[Bias-variance decomposition] \label{lemma:errbd}
Suppose the white-noise model of Section~\ref{subsec:problem_setting} holds, and let $\calC=(\mcA^\top\mcA)^{-s}$ with $s\ge0$. Let $\phi_* = \sum_{i} c_{i,*}\psi_i$ be the ground truth kernel, where $\{\psi_i\}$ are the eigenfunctions of $\mcA^\top \mcA$ associated with eigenvalues $\{\lambda_i\}$, and $c_{i,*} = \langle \psi_i, \phi_* \rangle$. The mean-squared error of the Tikhonov estimator decomposes as
\begin{equation}\label{eq:bv_decomp}
    \E\left[ \| \hat{\phi}_\lambda - \phi_{*} \|_2^2 \mid \mcA \right] =
    \underbrace{\sum_{i \ge 1} \frac{\lambda^4 c_{i,*}^2}{(\lambda_i^{s+1} + \lambda^2)^2}}_{\|e_{bias}\|_2^2}
    \;+\;
    \underbrace{\tfrac{\sigma^2}{M} \sum_{i \ge 1} \frac{\lambda_i^{2s+1}}{(\lambda_i^{s+1} + \lambda^2)^2}}_{\E[\|e_{var}\|_2^2\mid\mcA]}.
\end{equation}
Under Assumption~\ref{assump:operator_spectral_decay}, all eigenvalues satisfy $\lambda_i > 0$, and both series converge absolutely for any $\lambda > 0$ and $\phi_* \in L^2$.
\end{lemma}

\begin{proof}
Substituting $\bb= \mcA\phi_*+\bm\varepsilon $ into $\hat\phi_\lambda =(\mcA^\top \mcA+\lambda^2 \calC)^\dagger\mcA^\top \bb$,  we have
\begin{align*}
\hat{\phi}_\lambda - {\phi}_* &= \left( (\mcA^\top \mcA + \lambda^2 \calC)^\dagger \mcA^\top \mcA - I \right) \phi_* + (\mcA^\top \mcA + \lambda^2 \calC)^\dagger \mcA^\top  \bm\varepsilon := \eb_{bias}+\eb_{var}.
\end{align*}
The expectation of the squared error is
\begin{align*}
\E\left[ \| \hat{\phi}_\lambda - \phi_{*} \|_2^2 |\mcA \right]
& = \E\left[ \| e_{bias} \|_2^2 + 2 e_{bias}^T e_{var} + \| e_{var} \|_2^2|\mcA\right],
\end{align*}
where the cross-term vanishes because $e_{bias}$ is deterministic and
\[
\E[e_{var}|\mcA] = (\mcA^\top \mcA + \lambda^2 \calC)^\dagger \mcA^\top \E[\bm\varepsilon|\mcA] = 0.
\]

We evaluate the two terms on the eigenbasis. Under Assumption~\ref{assump:operator_spectral_decay} all $\lambda_i>0$, so $\calC\psi_i=\lambda_i^{-s}\psi_i$ and
\begin{align*}
(\mcA^\top \mcA + \lambda^2 \calC)\psi_i
&= \bigl(\lambda_i+\lambda^2\lambda_i^{-s}\bigr)\psi_i = \lambda_i^{-s}\bigl(\lambda_i^{s+1}+\lambda^2\bigr)\psi_i, \\
(\mcA^\top \mcA + \lambda^2 \calC)^\dagger \psi_i
&= \frac{\lambda_i^{s}}{\lambda_i^{s+1}+\lambda^2}\,\psi_i.
\end{align*}
For the bias, using $\mcA^\top\mcA\,\phi_*=\sum_i\lambda_i c_{i,*}\psi_i$,
\begin{align*}
e_{bias} &= \sum_{i\ge1}\Bigl(\frac{\lambda_i^{s+1}}{\lambda_i^{s+1}+\lambda^2}-1\Bigr)c_{i,*}\psi_i
= -\sum_{i\ge1}\frac{\lambda^2}{\lambda_i^{s+1}+\lambda^2}\,c_{i,*}\psi_i, \\
\|e_{bias}\|_2^2 &= \sum_{i\ge1}\frac{\lambda^4 c_{i,*}^2}{(\lambda_i^{s+1}+\lambda^2)^2}.
\end{align*}
For the variance, using the noise covariance $\E[\bm\varepsilon\bm\varepsilon^T ]=\tfrac{\sigma^2}{M} I $,
\begin{align*}
\E[\|e_{var}\|_2^2| \mcA]
&= \tfrac{\sigma^2}{M}\,\text{tr} \left( (\mcA^\top \mcA + \lambda^2 \calC)^\dagger \mcA^\top\mcA\, (\mcA^\top \mcA + \lambda^2 \calC)^\dagger \right) \\
&= \tfrac{\sigma^2}{M}\sum_{i\ge1}\frac{\lambda_i\cdot\lambda_i^{2s}}{(\lambda_i^{s+1}+\lambda^2)^2}
= \tfrac{\sigma^2}{M}\sum_{i\ge1}\frac{\lambda_i^{2s+1}}{(\lambda_i^{s+1}+\lambda^2)^2}.
\end{align*}
Adding the two contributions gives \eqref{eq:bv_decomp}. 
\end{proof}

We are now ready to prove Theorem~\ref{theorem:small_noise}.

\begin{proof}[\textbf{Proof of Theorem \ref{theorem:small_noise}}]
Throughout, $\beta$ is the smoothness exponent of Definition~\ref{def:smoothness}; by Lemma~\ref{lem:beta_classes}, $\beta=\tilde\beta$ in the Sobolev-ball case (a) and $\beta=\tilde\beta-\tfrac12$ in the power-law case (b). Only the bias bound is case-dependent; the variance and the bias-variance tradeoff depend on $\beta$ and $\gamma$ alone.

\smallskip
Set $\nu:=\beta/\gamma$ and $c_{i,*}:=\langle\phi_*,\psi_i\rangle$. By Lemma~\ref{lemma:errbd}, the conditional MSE for $\calC=(\mcA^\top\mcA)^{-s}$ splits into bias and variance,
\begin{equation}\label{eq:proof_decomp}
\mathbb{E}\bigl[\|\hat{\phi}_{\lambda}-\phi_{*}\|_{2}^{2}\mid\mcA\bigr]
= \underbrace{\sum_{i\ge1}\frac{\lambda^4 c_{i,*}^2}{(\lambda_i^{s+1}+\lambda^2)^2}}_{=:\,\|e_{bias}\|_2^2}
\;+\; \underbrace{\frac{\sigma^2}{M}\sum_{i\ge1}\frac{\lambda_i^{2s+1}}{(\lambda_i^{s+1}+\lambda^2)^2}}_{=:\,\E[\|e_{var}\|_2^2\mid\mcA]} .
\end{equation}
By Assumption~\ref{assump:operator_spectral_decay}, the eigenvalues obey the two-sided bound
\begin{equation}\label{eq:proof_eig}
\underline c\, i^{-\gamma}\le\lambda_i\le\overline c\, i^{-\gamma},
\qquad\text{hence}\qquad
\underline c^{\,\nu}\, i^{-\beta}\le \lambda_i^{\nu}\le\overline c^{\,\nu}\, i^{-\beta}.
\end{equation}
We bound the two terms of \eqref{eq:proof_decomp} in turn and then choose $\lambda$ to balance them.

\medskip
\noindent\textbf{Bias, Sobolev-ball case (a).}
Under Assumption~\ref{assump:source}(a), writing $c_{i,*}=i^{-\beta}\xi_i$ with the residual sequence $\xi_i:=i^\beta c_{i,*}$ gives $\sum_{i\ge1}\xi_i^2\le\mcE$ directly from \eqref{eq:class_sob}, and \eqref{eq:proof_eig} gives $i^{-\beta}\le\underline c^{-\nu}\lambda_i^\nu$. Substituting both into the bias sum in \eqref{eq:proof_decomp},

\begin{equation}\label{eq:bias_sob_sup}
\|e_{bias}\|_2^2
\;=\; \lambda^4\sum_{i\ge1}\frac{c_{i,*}^2}{(\lambda_i^{s+1}+\lambda^2)^2}
\;\le\; \underline c^{-2\nu}\,\mcE\,\lambda^4\,\sup_{i\ge1}\Bigl(\frac{\lambda_i^{\nu}}{\lambda_i^{s+1}+\lambda^2}\Bigr)^2 .
\end{equation}
Setting $x:=\lambda_i^{s+1}$ and $a:=\nu/(s+1)=\beta/[(s+1)\gamma]$ rewrites the inner ratio as a single-variable function,
\[
\frac{\lambda_i^{\nu}}{\lambda_i^{s+1}+\lambda^2} \;=\; h(\lambda_i^{s+1}), \qquad h(x):=\frac{x^{a}}{x+\lambda^2} .
\]
The exponent $a$ separates the two regimes of Theorem~\ref{theorem:small_noise}: $a<1$ corresponds to $\gamma>\beta/(s+1)$ (under-rough) and $a>1$ to $\gamma<\beta/(s+1)$ (over-rough). We bound $\sup_{i\ge1}h(\lambda_i^{s+1})$ in each.

If $a<1$, then $h(x)\to 0$ as $x\to 0^+$ and as $x\to\infty$, so $h$ has an interior maximum on $(0,\infty)$. Solving $h'(x)=0$ gives the critical point $x_\ast=\tfrac{a}{1-a}\lambda^2$, and
\[
\sup_{i\ge1}h(\lambda_i^{s+1}) \;\le\; \sup_{x>0}h(x) \;=\; h(x_\ast) \;=\; C_a\,\lambda^{2(a-1)},
\qquad C_a:=\Bigl(\tfrac{a}{1-a}\Bigr)^{a}(1-a) .
\]
If $a>1$, then $h'(x)>0$ for all $x>0$, so $h$ is strictly increasing on $(0,\infty)$. The supremum over the discrete set $\{\lambda_i^{s+1}\}_{i\ge1}\subset(0,\lambda_1^{s+1}]$ is therefore attained at $i=1$,
\[
\sup_{i\ge1}h(\lambda_i^{s+1}) \;=\; h(\lambda_1^{s+1}) \;\le\; \frac{(\lambda_1^{s+1})^{a}}{\lambda_1^{s+1}} \;=\; (\lambda_1^{s+1})^{a-1} \;=:\; C_a' ,
\]
which is independent of $\lambda$.

The constants $C_a,C_a'$ depend only on $s,\beta,\gamma$ (and $\lambda_1$ for $C_a'$). Substituting them back into \eqref{eq:bias_sob_sup} and using $2a=2\beta/[(s+1)\gamma]$ yields the two-regime bias bound
\begin{equation}\label{eq:bias_bound}
\|e_{bias}\|_2^2 \;\le\;
\begin{cases}
C_{bias}\,\lambda^{4\beta/[(s+1)\gamma]}, & \gamma>\beta/(s+1)\quad(\text{under-rough}),\\[4pt]
C_{bias}'\,\lambda^{4}, & \gamma<\beta/(s+1)\quad(\text{over-rough}).
\end{cases}
\end{equation}
At the boundary $\gamma=\beta/(s+1)$ (equivalently $a=1$, $\nu=s+1$), the two branches of \eqref{eq:bias_bound} have the same exponent $\lambda^{4a}|_{a=1}=\lambda^4$, and the supremum in \eqref{eq:bias_sob_sup} is bounded directly: $h(x)=x/(x+\lambda^2)$ is increasing with $h(x)\le 1$, so $\sup_{i\ge1}h(\lambda_i^{s+1})\le 1$ and $\|e_{bias}\|_2^2\le \underline c^{-2(s+1)}\,\mcE\,\lambda^4$. The Sobolev-ball bias bound therefore extends across $a=1$ without modification.

\medskip
\noindent\textbf{Bias, power-law case (b).}
Now the residual after factoring $i^{-\beta}$ is not square-summable, so we bound the bias sum in \eqref{eq:proof_decomp} directly. Assumption~\ref{assump:source}(b) and \eqref{eq:proof_eig} give
\[
c_{i,*}^2\le C_\phi\, i^{-(2\beta+1)},
\qquad
\lambda_i^{s+1}\ge \underline c^{\,s+1}\, i^{-g},
\qquad g:=(s+1)\gamma .
\]
Let $i_\ast:=\lambda^{-2/g}$ be the index at which $i^{-g}=\lambda^2$. Bounding the denominator below by $\underline c^{\,2(s+1)} i^{-2g}$ for $i\le i_\ast$ and by $\lambda^4$ for $i> i_\ast$,
\begin{equation}\label{eq:algb_split}
\|e_{bias}\|_2^2=\lambda^4\sum_{i\ge1}\frac{c_{i,*}^2}{(\lambda_i^{s+1}+\lambda^2)^2}
\;\le\; C\,\lambda^4\Bigl(\underbrace{\sum_{i\le i_\ast} i^{\,2g-2\beta-1}}_{(\mathrm{I})}
\;+\;\lambda^{-4}\underbrace{\sum_{i> i_\ast} i^{-(2\beta+1)}}_{(\mathrm{II})}\Bigr),
\end{equation}
with $C=C(C_\phi,\underline c,s)$. In the under-rough regime $\gamma>\beta/(s+1)$ we have $2g-2\beta-1>-1$ and $2\beta+1>1$, so $(\mathrm{I})$ is a partial sum of a divergent series (growing with $i_\ast$) while $(\mathrm{II})$ is the tail of a convergent series (vanishing with $i_\ast$). Applying the integral test to each,
\begin{align*}
(\mathrm{I}) \;=\; \sum_{i\le i_\ast} i^{\,2g-2\beta-1}
&\;\le\; \int_0^{i_\ast} x^{\,2g-2\beta-1}\,dx \;=\; \frac{i_\ast^{\,2(g-\beta)}}{2(g-\beta)} , \\
(\mathrm{II}) \;=\; \sum_{i> i_\ast} i^{-(2\beta+1)}
&\;\le\; \int_{i_\ast}^{\infty} x^{-(2\beta+1)}\,dx \;=\; \frac{i_\ast^{-2\beta}}{2\beta} .
\end{align*}
Substituting $i_\ast=\lambda^{-2/g}$ converts both bounds into powers of $\lambda$,
\begin{equation}\label{eq:algb_IandII}
(\mathrm{I})\;\le\; C\,i_\ast^{\,2(g-\beta)} \;=\; C\,\lambda^{-4+4\beta/g} ,
\qquad
(\mathrm{II})\;\le\; C\,i_\ast^{-2\beta} \;=\; C\,\lambda^{4\beta/g} .
\end{equation}
Plugging \eqref{eq:algb_IandII} into \eqref{eq:algb_split} gives
\[
\|e_{bias}\|_2^2 \;\le\; C\,\lambda^4\bigl(\lambda^{-4+4\beta/g}+\lambda^{-4}\cdot\lambda^{4\beta/g}\bigr)
\;=\; C\,\lambda^{4\beta/g} \;=\; C\,\lambda^{4\beta/[(s+1)\gamma]} .
\]
In the over-rough regime $\gamma<\beta/(s+1)$, however, $2g-2\beta-1<-1$, so $(\mathrm{I})\le\sum_{i\ge1} i^{2g-2\beta-1}\le C$ is uniformly bounded and $(\mathrm{II})$ remains bounded as before. The sum $(\mathrm{I})$ therefore dominates the bracket in \eqref{eq:algb_split}, yielding
\[
\|e_{bias}\|_2^2 \;\le\; C\,\lambda^4 .
\]
At the boundary $\gamma=\beta/(s+1)$ (equivalently $g=\beta$), the exponent in $(\mathrm{I})$ becomes exactly $2g-2\beta-1=-1$, so $(\mathrm{I})$ is the harmonic-type sum
\[
(\mathrm{I}) \;=\; \sum_{i\le i_\ast}\frac{1}{i} \;\le\; 1+\log i_\ast \;=\; 1 + \tfrac{2}{g}\,|\log\lambda| ,
\]
which grows logarithmically rather than as a power of $\lambda$. Plugging this into \eqref{eq:algb_split} together with the unchanged bound on $(\mathrm{II})$ gives
\[
\|e_{bias}\|_2^2 \;\le\; C\,\lambda^4\bigl(1+|\log\lambda|\bigr) ,
\]
so the boundary bias is the common value $\lambda^4$ of \eqref{eq:bias_bound} multiplied by an additional $|\log\lambda|$ factor.

Hence the bias bound \eqref{eq:bias_bound} holds for the power-law class $\mathcal{F}_{\mathrm{pow}}^{\tilde\beta}$ away from the boundary, with the additional $|\log\lambda|$ factor noted above at $\gamma=\beta/(s+1)$.

\medskip
\noindent\textbf{Variance.}
Factor the summand of the variance term in \eqref{eq:proof_decomp} as
\[
\frac{\lambda_i^{2s+1}}{(\lambda_i^{s+1}+\lambda^2)^2}
=\frac{\lambda_i^{s+1}}{\lambda_i^{s+1}+\lambda^2}\cdot\frac{\lambda_i^{s}}{\lambda_i^{s+1}+\lambda^2}.
\]
The second factor is maximized at $\lambda_i^{s+1}=s\lambda^2$, so
\[
\frac{\lambda_i^{s}}{\lambda_i^{s+1}+\lambda^2}\;\le\; c_s\,\lambda^{-2/(s+1)},
\qquad c_s=c_s(s),\ c_0=1 .
\]
Summing over $i$ and recognizing $\sum_{i}\lambda_i^{s+1}/(\lambda_i^{s+1}+\lambda^2)=d_{\rm eff}^{(s)}(\lambda)$, Corollary~\ref{cor:eff_dim_esti} gives
\begin{equation}\label{eq:var_gen}
\E[\|e_{var}\|_2^2\mid\mcA]
\;\le\; c_s\,\frac{\sigma^2}{M}\,\lambda^{-2/(s+1)}\,d_{\rm eff}^{(s)}(\lambda)
\;\le\; C_{var}\,\sigma^2\,\lambda^{-\frac{2}{s+1}\left(1+\frac{1}{\gamma}\right)},
\end{equation}
where $C_{var}=c_s\,C_2^{(s)}(\overline c,\gamma)/M$.

\medskip
\noindent\textbf{Bias-variance tradeoff.}
Adding \eqref{eq:bias_bound} and \eqref{eq:var_gen} and minimizing over $\lambda$ balances bias against variance. In the under-rough regime $\gamma>\beta/(s+1)$, we have
\[
\lambda_* \asymp \sigma^{\frac{(s+1)\gamma}{2\beta+\gamma+1}},
\qquad
\E[\|\hat\phi_{\lambda_*}-\phi_*\|_2^2\mid\mcA] \le C_1\,\sigma^{\frac{4\beta}{2\beta+\gamma+1}} .
\]
In the over-rough regime $\gamma<\beta/(s+1)$, where the bias saturates at $\lambda^4$,
\[
\lambda_* \asymp \sigma^{\frac{s+1}{(2s+3)+1/\gamma}},
\qquad
\E[\|\hat\phi_{\lambda_*}-\phi_*\|_2^2\mid\mcA] \le C_2\,\sigma^{\frac{4(s+1)}{(2s+3)+1/\gamma}},
\]
where the constants $C_1$ and $C_2$ depend on the sample size $M$, $\mcE$, $\beta$, $\gamma$, $s$, and the spectral constants $\underline{c},\overline{c}$ of Assumption~\ref{assump:operator_spectral_decay}. These are the two branches of \eqref{eq:opti_lambda_rate}.

At the boundary $\gamma=\beta/(s+1)$, the two rate exponents above coincide,
\[
\frac{4\beta}{2\beta+\gamma+1} \;=\; \frac{4(s+1)}{(2s+3)+1/\gamma} ,
\]
and the boundary rate is obtained as follows. For $\phi_*\in\mathcal{F}_{\mathrm{Sob}}^{\tilde\beta}$, the bias bound \eqref{eq:bias_bound} extends to $a=1$ unchanged, so the common rate is attained as stated. For $\phi_*\in\mathcal{F}_{\mathrm{pow}}^{\tilde\beta}$, the boundary bias is $\|e_{bias}\|_2^2\le C\lambda^4|\log\lambda|$; balancing this against \eqref{eq:var_gen} with the dominant power-of-$\lambda$ scaling gives
\[
\lambda_* \asymp \sigma^{\frac{s+1}{(2s+3)+1/\gamma}},
\qquad
\E[\|\hat\phi_{\lambda_*}-\phi_*\|_2^2\mid\mcA] \le C_3\,\sigma^{\frac{4\beta}{2\beta+\gamma+1}}\,|\log\sigma| ,
\]
matching the non-boundary rate of \eqref{eq:opti_lambda_rate} up to a multiplicative $|\log\sigma|$. For consistency, since Theorem~\ref{theorem:small_noise} states the rate only on $\gamma\ne\beta/(s+1)$, we do not include this boundary case in its statement and leave it aside in what follows.
\end{proof}

\end{document}